\documentclass[11pt]{amsart}
\usepackage[T1]{fontenc}
\usepackage{lmodern}
\usepackage{amssymb}
\usepackage{mathtools}
\usepackage{booktabs}
\usepackage{array}
\usepackage{microtype}
\usepackage[colorlinks=true,linkcolor=blue,citecolor=blue,urlcolor=blue]{hyperref}
\hypersetup{
 pdftitle={Coefficientwise Extremality and Rigidity for Independent Sets in Bipartite Graph Covers},
 pdfauthor={Yusuke Watanabe},
 pdfsubject={Fixed-occupation independent-set counts in bipartite graph covers},
 pdfkeywords={independent sets, graph covers, holonomy, treewidth, chordal bipartite}
}

\newtheorem{mainthm}{Theorem}

\newtheorem{theorem}{Theorem}[section]
\newtheorem{proposition}[theorem]{Proposition}
\newtheorem{lemma}[theorem]{Lemma}
\newtheorem{corollary}[theorem]{Corollary}
\theoremstyle{definition}

\newtheorem{example}[theorem]{Example}
\theoremstyle{remark}
\newtheorem{remark}[theorem]{Remark}
\numberwithin{equation}{section}

\newcommand{\Pa}{\mathsf P}
\newcommand{\Za}{\mathsf Z}
\newcommand{\Da}{\mathsf D}
\newcommand{\Ra}{\mathsf R}
\newcommand{\occ}{\mathbf a}
\newcommand{\one}{\mathbf 1}
\newcommand{\ind}[1]{[#1]}
\newcommand{\Om}{\Omega}
\newcommand{\NN}{\mathcal N}
\newcommand{\CC}{\mathcal C}
\newcommand{\id}{\mathrm{id}}
\newcommand{\tw}{\operatorname{tw}}
\newcommand{\Fix}{\operatorname{Fix}}
\newcommand{\tr}{\operatorname{tr}}
\newcommand{\hol}{\operatorname{hol}}
\newcommand{\coefge}{\succeq_{\mathrm{coeff}}}
\newcommand{\coefle}{\preceq_{\mathrm{coeff}}}
\newcommand{\xx}{\mathbf x}

\title[Independent sets in bipartite graph covers]
{Coefficientwise Extremality\\
and Rigidity for Independent Sets\\
in Bipartite Graph Covers}
\author{Yusuke Watanabe}
\email{ywatanabex@gmail.com}
\date{October 10, 2026}

\begin{document}
\begin{abstract}
Let \(G\) be a finite bipartite graph and let \(\widetilde G\) be an
\(M\)-fold cover. We compare independent sets of \(\widetilde G\) with those
of the trivial cover, consisting of \(M\) disjoint copies of \(G\), after
prescribing how many vertices are selected above each vertex of \(G\).
A conjecture from 2011, motivated by the Bethe approximation, asserts
that the trivial cover maximizes every such count.

We prove this for every feasible occupation profile whose active graph
has treewidth at most two or is chordal bipartite. Equality holds exactly
when the restricted cover is flat, meaning that fibre relabelings make it
trivial. For \(M\ge2\) and positive feasible occupations, a nonflat cover
of a connected bipartite graph of treewidth at most two has a deficit
of at least twice the coefficient of \(p_G^{M-2}\) with every
occupation reduced by one. A holonomy-dependent bound improves the factor
to three for odd-order holonomy; both uniform constants are sharp.
Comparison and rigidity survive finite compositions along an edge or a
vertex, and quantitative bounds from one piece propagate to the whole
graph, giving results beyond the two original classes. On any bipartite
base, the first nonzero homogeneous defect is determined by the shortest
cycles with nontrivial holonomy.

At singleton occupations, our counts are numbers of DP-colorings of full
covers. The dual DP color function on these classes and their finite seam
compositions equals the chromatic polynomial, with equality attained
precisely by canonical labelings. Dong's chromatic-polynomial inequality
then yields the Dual-DP Shameful inequality on the same graph class.
The main theorems have been formalized in Lean~4.
\end{abstract}
\maketitle

\section{Introduction}\label{sec:intro}

\subsection{The question}
Let \(G=(L,R,E)\) be a finite simple bipartite graph and let \(M\ge1\).
An \emph{\(M\)-cover} \(\widetilde G\) of \(G\), also called an
\emph{\(M\)-fold cover}, has vertex set
\(V(G)\times[M]\) and, for every edge \(uv\in E\), a perfect matching
between the fibres \(\{u\}\times[M]\) and \(\{v\}\times[M]\). We record
the matching above \(uv\) by a permutation \(\alpha_{uv}\in S_M\), joining
\((u,i)\) to \((v,\alpha_{uv}(i))\), with \(\alpha_{vu}=\alpha_{uv}^{-1}\);
this is the permutation-voltage description of covers \cite{GrossTucker}.
We also call \(\alpha_{uv}\) the \emph{edge transport}.
The simplest cover is the \emph{trivial} one, in which every
\(\alpha_{uv}\) is the identity and \(\widetilde G\) is \(M\) disjoint
copies of \(G\).

Every independent set of \(\widetilde G\) projects to \(G\), and the
projection records, above each vertex \(v\), how many of the \(M\) lifts of
\(v\) were selected. We call this number the \emph{occupation} of \(v\),
and we ask the following question: for a fixed occupation at every vertex,
which cover has the most independent sets? To make this precise, let
\[
 p_G(\xx)=\sum_{I\subseteq V(G)\text{ independent}}\ \prod_{v\in I}x_v
\]
be the multivariate independent-set polynomial of \(G\), and let
\(\Pi_{\widetilde G}(\xx)\) be the independent-set polynomial of
\(\widetilde G\) after the substitution \(x_{(v,i)}=x_v\). The trivial
cover has polynomial \(p_G^M\). For a vector of nonnegative integers
\(\occ=(a_v)_{v\in V(G)}\) we write
\begin{equation}\label{eq:counts}
 \Pa_G(\occ)=[\xx^\occ]\,p_G^M,\qquad
 \Za_{\widetilde G}(\occ)=[\xx^\occ]\,\Pi_{\widetilde G},\qquad
 \Da_{\widetilde G}(\occ)=\Pa_G(\occ)-\Za_{\widetilde G}(\occ),
\end{equation}
for the number of independent sets with occupation profile \(\occ\) in the
trivial cover, the same number in \(\widetilde G\), and their difference,
which we call the \emph{deficit}. We call \(\Pa\) the \emph{ordinary count}
and \(\Za\) the \emph{gained count}. Thus the comparison is made at the level
of the whole vector of occupations, not merely the total size of an
independent set.

The conjecture of \cite[Conjecture~1, equation~(4)]{Watanabe}, proposed
in 2011, asserts that the trivial cover always wins:
\[
 \Pi_{\widetilde G}\ \coefle\ p_G^M
 \qquad\text{for every finite bipartite \(G\) and every cover
 \(\widetilde G\),}
\]
where \(\coefle\) denotes coefficientwise inequality. A single twisted edge
already shows what is at stake. Take \(G=C_4\), \(M=3\), the identity on
three edges and the transposition \((1\,2)\) on the fourth. For the
all-ones profile \(\occ=\one\), also called the \emph{singleton profile},
\(\Pa=18\) and \(\Za=16\); for the profile that selects
two lifts above two opposite vertices and one above the other two,
\(\Pa=3\) and \(\Za=1\). In both cases the twist costs exactly two
independent sets, and the second case illustrates multiple selections
within a fibre (Example~\ref{ex:square}).

Two communities have reasons to care about this question. In statistical
physics and coding theory, Vontobel characterized the Bethe partition
function by a limit superior of \(M\)-th roots of average partition
functions over \(M\)-covers \cite[Definition~32 and Theorem~33]{Vontobel}, so that a
comparison valid for every cover yields a bound on the Bethe approximation;
the conjecture above was proposed as a coefficientwise strengthening of such
a bound. In graph coloring, when every occupation equals one, the count
\(\Za_{\widetilde G}(\one)\) is the number of DP-colorings (correspondence
colorings) of \(G\) with respect to the full cover obtained from
\(\widetilde G\), and \(\Pa_G(\one)\) is the chromatic polynomial
\(P(G,M)\). The conjecture therefore contains the statement that no full
cover of a bipartite graph admits more DP-colorings than the chromatic
polynomial predicts, a question studied under the name \emph{dual DP color
function} \cite{Mudrock,KMS,ThetaDP}. We return to both connections in
Sections~\ref{sec:intro-dp} and~\ref{sec:intro-context}.

The conjecture is open in general. This paper proves it, with a
characterization of equality and a sharp quantitative version, on two
structural classes of graphs, and shows that these conclusions are stable
under gluing.

\subsection{Main results}\label{sec:intro-main}
A profile \(\occ\) is \emph{feasible} if
\begin{equation}\label{eq:feasible}
 0\le a_v\le M\quad(v\in V(G)),\qquad a_u+a_v\le M\quad(uv\in E);
\end{equation}
these are exactly the profiles with \(\Pa_G(\occ)>0\)
(Lemma~\ref{lem:counting}). The \emph{active graph} of \(\occ\) is the
induced subgraph \(G_\occ=G[\{v:a_v>0\}]\); vertices of occupation zero
impose no constraint, so only the cover restricted to \(G_\occ\) matters.
We call \(\widetilde G|_{G_\occ}\) the \emph{active cover}.

Vertexwise relabelings do not change the cover up to isomorphism.
A cover is \emph{flat} if such relabelings make every edge permutation
the identity, equivalently if the composite permutation around every
cycle, its \emph{holonomy}, is the identity.
In the language of gain graphs, relabeling is \emph{switching} and flatness
is \emph{balance} \cite[Section~5]{Zaslavsky}; in the language of
DP-coloring, a flat cover is one admitting a \emph{canonical labeling}:
each fibre is labeled by \([M]\) so that every edge matching joins equal
labels at its endpoints.
A flat cover is isomorphic to the trivial one, so it has
\(\Za_{\widetilde G}=\Pa_G\) at every profile. Our first theorem says that
on two classes of graphs the converse holds as well: the trivial cover is
the unique maximizer on the active graph, up to switching.

Recall that a bipartite graph is \emph{chordal bipartite} if it has no
induced cycle of length at least six, and that treewidth is defined through
tree decompositions (Section~\ref{sec:networks}). Throughout,
``width two'' means treewidth \emph{at most} two.

\begin{mainthm}[Structural extremality and rigidity]\label{thm:A}
Let \(M\ge1\), let \(\widetilde G\) be an \(M\)-cover of a finite simple
bipartite graph, and let \(\occ\) be feasible. If \(G_\occ\) has treewidth
at most two or is chordal bipartite, then
\[
 \Za_{\widetilde G}(\occ)\le\Pa_G(\occ),
 \qquad
 \Za_{\widetilde G}(\occ)=\Pa_G(\occ)
 \ \Longleftrightarrow\
 \widetilde G|_{G_\occ}\text{ is flat}.
\]
In particular either condition on \(G\) implies
\(\Pi_{\widetilde G}\coefle p_G^M\).
\end{mainthm}

The proof uses only the active graph, since zero occupations impose no
constraints. Both classes are hereditary, so either hypothesis on \(G\)
implies the same hypothesis on its active graph. The classes are
incomparable: \(C_6\) has
treewidth two but is not chordal bipartite, whereas \(K_{3,3}\) is chordal
bipartite and has treewidth three.

The second theorem measures how much a twist costs. The natural unit is
the corresponding coefficient on a cover with two fewer sheets. For \(M\ge2\) define
the \emph{reduced reserve}
\begin{equation}\label{eq:reserve}
 \Ra_G(\occ)=[\xx^{\occ-\one}]\,p_G^{M-2},
\end{equation}
the number of independent sets of the trivial \((M-2)\)-cover in which
every occupation has been lowered by one; a negative exponent has
coefficient zero and \(p_G^0=1\). The constant in front of it depends
only on the holonomy of the cover. For a nontrivial subgroup
\(\Gamma\le S_M\) let \(m_k(\gamma)\) be the number of \(k\)-subsets of
\([M]\) that \(\gamma\) does not fix, and put
\begin{equation}\label{eq:kappa}
 \kappa(\Gamma)=
 \min_{\substack{\gamma\in\Gamma\setminus\{1\}\\
                 1\le k\le\lfloor M/2\rfloor}}
 \frac{m_k(\gamma)}{\binom{M-2}{k-1}}.
\end{equation}
The \emph{holonomy image} of a connected cover is the subgroup of \(S_M\)
generated by the edge permutations after a spanning tree has been
relabeled to the identity; it is determined up to conjugation, hence so is
\(\kappa\). Note that \(\kappa\) depends on the permutation action on
\([M]\), not only on the abstract group.

\begin{mainthm}[Sharp deficit on treewidth two]\label{thm:B}
Suppose \(G\) is connected and bipartite, has an edge, and
\(\tw(G)\le2\). Let \(M\ge2\), let \(\widetilde G\) be nonflat with
holonomy image \(\Gamma\), and let \(\occ\) be positive and feasible. Then
\[
 \Da_{\widetilde G}(\occ)\ge\kappa(\Gamma)\Ra_G(\occ)>0.
\]
Always \(\kappa(\Gamma)\ge2\). If \(|\Gamma|\) is odd, then
\(\kappa(\Gamma)\ge3\). The uniform constants two and three are sharp.
\end{mainthm}

Sharpness of the constant two holds for the entire polynomial: if an even
cycle carries one transposition and identities elsewhere, then
\[
 p_G^M-\Pi_{\widetilde G}=2\,\xx^{V}p_G^{M-2}
\]
(Example~\ref{ex:sharp}). Odd order is sufficient but not necessary for
the constant three; the subgroup of \(S_6\) generated by
\((1\,2)(3\,4)(5\,6)\) has \(\kappa=3\), attained at a saturated profile on
\(C_4\) (Example~\ref{ex:even}). The proof actually gives a
profile-dependent constant, in which the minimum in \eqref{eq:kappa} runs
only over \(k\le\max_v\min(a_v,M-a_v)\); at the all-ones profile this
constant is the minimal degree of \(\Gamma\), the least number of labels
moved by a nonidentity element (Corollary~\ref{cor:profilekappa}).

The third theorem shows that the first two are stable under gluing. Call a
graph a \emph{finite seam composition} if it is built from bipartite
\emph{atoms}, each of treewidth at most two or chordal bipartite, by
repeatedly taking the union of two pieces whose intersection is a single
edge with its endpoints, a single vertex, or empty, with no further edges
between the pieces. The bipartitions, occupations, and edge permutations agree
on the intersection.
In particular, the two pieces use the same permutation on a shared
oriented edge in the same endpoint coordinates, and that edge is included
once. A finite construction tree records the atoms and seams. Such graphs
need belong to neither atom class: gluing
\(K_{n,n}\) with \(n\ge3\) to \(C_6\) along an edge produces a graph that is
not chordal bipartite and has treewidth \(n\) (Example~\ref{ex:mixed}).
For a finite seam composition \(H\) with an \(M\)-cover, we summarize
Theorem~\ref{thm:C} in the following form; see
Section~\ref{sec:composition} for its precise statement:
\begin{enumerate}
\renewcommand{\labelenumi}{(\roman{enumi})}
\item \(\Za_{\widetilde H}(\occ)\le\Pa_H(\occ)\) for every profile, and
 more strongly, the completion count of any legal selection at a vertex
 or an edge is at most its trivial-cover count, after relabeling a pinned
 edge to match equal labels;
\item at a feasible profile, equality holds exactly when the cover is flat
 on the active graph, and the relative deficit
 \(\Da/\Pa\) of \(H\) is at least that of every atom and every intermediate
 piece;
\item if \(M\ge2\), the profile is feasible with all occupations
 positive, and a designated atom occurrence \(W\) satisfies
 \(\Da_W\ge\lambda\Ra_W\) for \(\lambda>0\), then
 \(\Da_H\ge\lambda\Ra_H\). In particular one connected nonflat width-two
 atom with
 holonomy image \(\Gamma_W\) forces \(\Da_H\ge\kappa(\Gamma_W)\Ra_H\),
 irrespective of the holonomy on the other atoms.
\end{enumerate}

Finally, on an arbitrary bipartite base, the deficit detects a twist in its
lowest-degree part. If the cover is nonflat, let \(g\) be the length of a shortest cycle with
nonidentity holonomy, and for such a cycle \(C\) let \(m_C\) be the number
of labels moved by its holonomy. Then (Proposition~\ref{prop:first})
\[
 (p_G^M-\Pi_{\widetilde G})_{[d]}=0\ \ (d<g),\qquad
 (p_G^M-\Pi_{\widetilde G})_{[g]}
 =\sum_{\substack{|C|=g\\ \hol(C)\ne1}}m_C\prod_{v\in V(C)}x_v,
\]
where the subscript denotes the homogeneous part of that degree. In
particular \(p_G^M=\Pi_{\widetilde G}\) if and only if the cover is flat,
for every bipartite \(G\). What Theorems~\ref{thm:A}--\ref{thm:C} add is
that, on the stated classes, each feasible coefficient detects
nonflatness of its active cover;
Theorem~\ref{thm:B} additionally supplies a sharp lower bound for
positive profiles on connected width-two bases.

\subsection{Consequences for DP-coloring}\label{sec:intro-dp}
DP-coloring, also called correspondence coloring, was introduced by
Dvo\v{r}\'{a}k and Postle \cite{DvorakPostle}. Adding all edges inside each
fibre of \(\widetilde G\) produces a \emph{full} \(M\)-fold DP cover
\(\mathfrak H\), and an independent transversal of \(\mathfrak H\), one
vertex per fibre, is precisely an independent set of \(\widetilde G\) with
the all-ones profile. Writing \(P_{DP}(G,\mathfrak H)\) for the number of
such transversals, we have \(P_{DP}(G,\mathfrak H)=\Za_{\widetilde G}(\one)\)
and \(P(G,M)=\Pa_G(\one)\). Kaul and Mudrock \cite{KaulMudrock} introduced
the DP color function \(P_{DP}(G,M)\), the minimum of
\(P_{DP}(G,\mathfrak H)\) over \(M\)-fold covers, and Mudrock \cite{Mudrock}
the \emph{dual DP color function}
\[
 P^*_{DP}(G,M)=\max_{\mathfrak H\text{ full }M\text{-fold}}
 P_{DP}(G,\mathfrak H)\ \ge\ P(G,M).
\]
Restricting to full covers is essential here: the edge matchings must be
bijections, as in our permutation model. Kaul, Mudrock, and Sharma proved
\(P^*_{DP}(K_{m,n},M)=P(K_{m,n},M)\) \cite[Theorem~10]{KMS}, and Bui,
Kaul, Maxfield, Mudrock, Shin, and Thomason obtained the same identity for
bipartite generalized theta graphs \cite[Theorem~14]{ThetaDP}. Since
\(K_{m,n}\) is chordal bipartite and bipartite generalized theta graphs
have treewidth at most two, Theorem~\ref{thm:A} at the all-ones profile
contains both results and adds the characterization of the maximizers.

\begin{samepage}
\begin{corollary}[Dual DP color function]\label{cor:dualDP}
Let \(H\) be a finite seam composition. Then, for every \(M\ge1\),
\[
 P^*_{DP}(H,M)=P(H,M).
\]
A full \(M\)-fold cover attains this maximum if and only if it admits a
canonical labeling. In particular this holds for every bipartite graph
of treewidth at most two and every chordal bipartite graph.
\end{corollary}
\end{samepage}
\begin{proof}
For \(M\ge2\), the singleton profile is positive and feasible, and its
active graph is \(H\). Apply Theorem~\ref{thm:C}(i)--(ii); the identity cover
attains the ordinary count. For \(M=1\), every edge permutation is the
identity, so every full cover is canonical and its count is \(P(H,1)\),
including the zero count when \(H\) has an edge.
\end{proof}

Combined with Dong's inequality
\(P(H,k+1)/(k+1)^n\ge P(H,k)/k^n\) for positive integers \(k\ge n-1\)
\cite[Theorem~2.2]{Dong}, this settles a conjecture of Kaul, Mudrock, and
Sharma on these graphs.

\begin{corollary}[Normalized dual DP inequality]\label{cor:shameful}
Let \(H\) be a finite seam composition on \(n\ge1\) vertices. For every
integer \(k\ge\max\{1,n-1\}\),
\[
 \frac{P^*_{DP}(H,k+1)}{(k+1)^n}
 \ge \frac{P^*_{DP}(H,k)}{k^n}.
\]
In particular this holds on both atom classes of Theorem~\ref{thm:A}.
\end{corollary}
\begin{proof}
Corollary~\ref{cor:dualDP} gives \(P^*_{DP}(H,j)=P(H,j)\) for every
\(j\ge1\).
Dong's chromatic-polynomial inequality \cite[Theorem~2.2]{Dong} states that
\(P(H,k+1)/(k+1)^n\ge P(H,k)/k^n\) for positive integers
\(k\ge n-1\). Substitution proves the assertion.
\end{proof}
This establishes the Dual-DP Shameful inequality of
\cite[Conjecture~9]{KMS} on the stated graph class; the general
conjecture remains open. The complete-bipartite case
\cite[Corollary~11]{KMS} was obtained there by the same use of Dong's
theorem. Theorem~\ref{thm:B} adds a quantitative statement in this setting:
for \(M\ge2\), a noncanonical full \(M\)-fold cover
\(\mathfrak H\) of a connected width-two bipartite graph with an edge
and holonomy image \(\Gamma\) satisfies
\[
 P(G,M)-P_{DP}(G,\mathfrak H)\ge\operatorname{mindeg}(\Gamma),
\]
where \(\operatorname{mindeg}(\Gamma)\) is the least number of labels
moved by a nonidentity element (Corollary~\ref{cor:profilekappa}).

\subsection{Relation to earlier work}\label{sec:intro-context}
\paragraph{Evaluations versus coefficients.}
Ruozzi proved that for log-supermodular graphical models the partition
function of every \(M\)-cover is at most the \(M\)-th power of the base
partition function, which yields the Bethe lower bound
\cite[Theorem~4.1 and Section~4.2]{Ruozzi}. After complementing one side of
the bipartition, independent sets of a bipartite graph form such a model,
so \(\Pi_{\widetilde G}(\boldsymbol\lambda)\le p_G(\boldsymbol\lambda)^M\)
for all positive activities \(\boldsymbol\lambda\). Such evaluation
inequalities do not imply coefficient inequalities, as \((x-1)^2\ge0\)
shows, and a fixed-occupation constraint cannot be inserted into Ruozzi's
argument: the indicator of \(\{|S|=a\}\) on \(2^{[M]}\) is not
log-supermodular for \(0<a<M\), since two distinct \(a\)-sets have
intersection and union of other sizes. The coefficientwise question thus
requires different tools.

\paragraph{Two sheets.}
For \(M=2\) the conjecture is known for every bipartite base. Csikv\'ari's
comparison for the bipartite double cover works within fibres of the
projection \cite[proof of Theorem~2.1, equation~(2.8)]{Csikvari}, and since
the bipartite
double cover of a bipartite graph is two copies of it, that argument keeps
the whole occupation profile. Remark~\ref{rem:two} gives an elementary
proof which also yields rigidity and the exact deficit at \(M=2\).
The all-\(M\) results require elimination schemes that preserve mixed
occupations and arbitrary holonomy, beyond this two-sheet count.

\paragraph{DP color functions.}
Canonical-labeling rigidity for singleton counts was studied for unicyclic
and theta graphs by Li and Yang \cite[Theorems~3.9--3.10]{LiYang}; their
Examples~3.2--3.3 exhibit noncanonical full covers of nonbipartite graphs
with the ordinary number of colorings, so a value comparison and an
equality classification are separate questions. Outside the bipartite
world the dual function need not even be eventually polynomial: Mudrock and
Sharbel determine a parity-dependent value for \(K_4\)
\cite[Theorem~3]{DualPolynomial}. At the opposite extremum, Kaul, Mudrock,
Sharma, and Ullyot determine the minimum DP color function of a bipartite
graph for large \(q\) in terms of its Tutte polynomial
\cite[Theorem~3 and Remark~4]{KMSU}.
For an \(n\)-vertex bipartite graph containing a cycle, let \(g_0\) be its
girth, the length of a shortest cycle, and let \(t\) count its
\(g_0\)-cycles. They prove \cite[Theorem~6]{KMSU} that
\[
 P(G,q)-P_{DP}(G,q)=tq^{n-g_0+1}+O(q^{n-g_0})
 \qquad(q\to\infty).
\]
The same
shortest-cycle mechanism appears in Proposition~\ref{prop:first}, there as
an exact multivariate formula for a fixed cover and fixed \(M\).
Occupations larger than one occur in fractional DP-coloring: the
quasi-independent sets of Bernshteyn, Kostochka, and Zhu
\cite[Definitions~1.1, 1.4--1.5]{BKZ} with \(|Q\cap L(v)|=a_v\) are counted
exactly by \(\Za_{\widetilde G}(\occ)\). Their work asks for existence
under lower bounds on occupations; ours fixes every occupation and compares
counts.

\subsection{Methods and organization}\label{sec:intro-methods}
The two classes are handled by two different elimination schemes, each
designed to preserve the obstruction to flatness while removing a vertex.

For width two we recast the count as a sum over families of subsets
\(A_v\subseteq[M]\), one per vertex, with the edge constraint
\(\alpha_{uv}A_u\subseteq A_v\). This constraint depends only on the size of
\(A_u\cap\alpha_{uv}^{-1}A_v\). After subtracting the smallest possible
overlap, its indicator is a binomial coefficient of top degree.
Functions of this shifted overlap that are
nonnegative combinations of binomial coefficients form a cone that is closed
under the operation of summing out a vertex of degree two
(Lemma~\ref{lem:serial}); this is the exact step. When summing out creates
an edge parallel to an existing one, the product of two kernels is replaced
by a sum of two kernels that agrees with the product on the diagonal and
dominates it off the diagonal (Lemma~\ref{lem:split}). This step bounds
the cover count from above while keeping the trivial-cover count exact.
The quantitative induction retains a nonflat branch; when both branches
are flat, their different transports remain across a common bridge.
The loss there is computed exactly by harmonic analysis on the Johnson
scheme (Lemma~\ref{lem:local}), which is where \(\kappa(\Gamma)\)
comes from. The same induction, run with forward differences, controls the
reduced reserve and gives Theorem~\ref{thm:B}.

For chordal bipartite graphs we use the classical fact that these are the
incidence graphs of \(\beta\)-acyclic hypergraphs
\cite{Fagin,Ausiello,BraultBaron}: one can repeatedly delete a column whose
incident row neighborhoods are nested, a \emph{nest point}
\cite{BrouwerKolen}. Summing out such a column is a compression inequality
of rearrangement type (Lemma~\ref{lem:compression}), and it is an equality
in the trivial cover precisely because the relevant intersections are
nested. Rigidity is obtained by saturating one side of the bipartition and
contracting a tight star (Lemmas~\ref{lem:saturation}--\ref{lem:parallel}).
The following table summarizes the two mechanisms.
\begin{center}
\small
\begin{tabular}{@{}>{\raggedright\arraybackslash}p{.17\textwidth}
>{\raggedright\arraybackslash}p{.36\textwidth}
>{\raggedright\arraybackslash}p{.36\textwidth}@{}}
\toprule
 & Treewidth at most two & Chordal bipartite\\
\midrule
Reduction & Eliminate a vertex of degree at most two
 & Eliminate a nest column; contract a saturated star\\
Exact step & Serial convolution and its forward difference
 & Compression on nested intersections; star assignment bijection\\
Preserved loss & Diagonal-exact split or local holonomy collision
 & Nonnegative saturation remainder and rejected parallel constraints\\
Terminal case & Two flat branches separated by a bridge
 & A flat representative with a conflicting omitted edge\\
\bottomrule
\end{tabular}
\end{center}

Composition (Section~\ref{sec:composition}) rests on an amplification
argument: gluing many copies of a piece along an edge and applying
Theorem~\ref{thm:A} to the result shows that, in a piece from either class,
the number of completions of any legal selection on an edge is at most the
trivial-cover constant. Such fibre bounds pass through edge and vertex
seams, and a seam identity converts them into the equality characterization
and the propagation of deficits.

Section~\ref{sec:prelim} fixes notation, proves the first-defect formula,
and works out the twisted square. Sections~\ref{sec:networks}
and~\ref{sec:width} prove the width-two comparison and the sharp deficit;
Section~\ref{sec:chordal} proves the chordal-bipartite case;
Section~\ref{sec:composition} proves Theorem~\ref{thm:C};
Section~\ref{sec:questions} collects open questions. All inputs to
Theorems~\ref{thm:A}--\ref{thm:C} are proved in the paper. The harmonic
analysis on slices of the Boolean cube is classical \cite{Filmus,Ghareghani},
but we need it with a specific counting normalization and a monotonicity of
its scalars, so Appendix~\ref{app:spectral} derives what is used. Likewise
Appendix~\ref{app:structure} follows Brault-Baron's proof of the
two-nest-point theorem \cite{BraultBaron} to obtain a nest-point order
ending at a prescribed column, which the rooted comparisons require.
This root-last choice is a consequence of the classical two-nest-point
theorem, rather than a new structural result. Corollary~\ref{cor:shameful}
uses Dong's inequality in addition to these self-contained inputs.

\begin{samepage}
\subsection*{Formalization}
All clauses of Theorems~\ref{thm:A}--\ref{thm:C} have been formalized in
Lean~4 using mathlib. The source code, build instructions, and a
correspondence between the statements in this paper and the formal
declarations are available in the accompanying repository
\cite{WatanabeLean}. The repository documents the precise scope of the
formalization; some auxiliary results and applications are not included.
\par
\end{samepage}

\begin{samepage}
\section{Occupations, gauges, and the first defect}\label{sec:prelim}

We use the following terminology throughout: the \emph{ordinary count}
\(\Pa\) refers to the trivial cover, and the \emph{gained count} \(\Za\)
to the cover with its prescribed edge permutations, or \emph{transports}.
A family of vertexwise relabelings is called a \emph{gauge}.
The \emph{active cover} is the restriction to \(G_\occ\).
The singleton and all-ones profiles both mean \(\occ=\one\), and
\(M\)-cover and \(M\)-fold cover are synonymous.
\end{samepage}

We use \(\binom nk=0\) for \(n\ge0\) unless \(0\le k\le n\);
\(\binom00=1\). Empty sums are zero and empty products are one.
There is one assignment on the empty vertex set, so an assignment sum
over no vertex variables has one term.
All independent-set base graphs are finite, simple, and bipartite.
Auxiliary overlap networks below may have triangles, and contracted
constraint graphs may have parallel labelled edges.
When the inherited profile and cover are clear, we abbreviate
\(\Pa_J,\Za_J,\Da_J,\Ra_J\) by \(P_J,Z_J,D_J,R_J\).
The following notation distinguishes raw occupations from inclusion ranks
and boundary constants from saturation weights.
\begin{center}
\small
\begin{tabular}{@{}ll@{}}
\toprule
Symbol & Meaning\\
\midrule
\(\occ,\ a_v\) & Raw occupation vector and its vertex entries\\
\(p_v\) & Inclusion rank: \(a_v\) on \(L\), \(M-a_v\) on \(R\)\\
\(\ell_{pq},r_{pq},z\) & Minimum overlap, shifted radius, shifted overlap\\
\(P,Z,D,R\) & Ordinary count, gained count, deficit, reduced reserve\\
\(\Gamma,\kappa(\Gamma)\) & Holonomy image and its moved-subset constant\\
\(\beta,c\) & Legal seam cardinality and ordinary fibre constant\\
\(b_v,\eta_v,\Lambda,\varepsilon\) & Saturated occupation, weight, product, surplus\\
\(\partial_k,\mathcal H_i\) & Down operator and harmonic subspace\\
\(\overline H\) & Underlying simple graph of a contracted multigraph\\
\bottomrule
\end{tabular}
\end{center}
The independent-set polynomial \(p_G\), inclusion ranks \(p_v\), and
chromatic polynomial \(P(G,M)\) have different subscripts or arguments;
\(\Pa_G(\occ)\) always denotes the fixed-occupation count.
An edge \emph{transport} is its permutation map (also called a gain in
gain-graph terminology). A \emph{holonomy} is the composite transport
around a closed walk; it is not an individual edge label.

\begin{lemma}[Counting and support]\label{lem:counting}
\(\Za_{\widetilde G}(\occ)\) counts families \(S_v\subseteq[M]\) with
\(|S_v|=a_v\) and
\[
 S_v\cap\alpha_{uv}S_u=\varnothing\quad(uv\in E).
\]
\(\Pa_G(\occ)\) is the same count with every map identity.
Outside \eqref{eq:feasible} both counts vanish; within it \(\Pa_G(\occ)>0\).
Removing zero occupations preserves both counts.
An isolated active vertex supplies the common positive factor
\(\binom M{a_v}\).
\end{lemma}
\begin{proof}
Selecting the vertices above \(v\) is exactly selecting \(S_v\); the
displayed condition excludes precisely the lifted edges. The identity
cover consists of \(M\) independent copies, proving the ordinary formula.
Disjointness gives the feasibility inequalities in either model.
For a feasible profile choose \(S_u=[a_u]\) on \(L\) and
\(S_v=[M]\setminus[M-a_v]\) on \(R\); every edge is disjoint.
A zero occupation has a unique empty set and vacuous constraints;
an isolated vertex is unrestricted.
\end{proof}

In inclusion coordinates write \(A_u=S_u\) on \(L\) and
\(A_v=[M]\setminus S_v\) on \(R\), with ranks
\begin{equation}\label{eq:inclusion}
 p_u=a_u,\quad p_v=M-a_v,\qquad \alpha_{uv}A_u\subseteq A_v.
\end{equation}
In a feasible active component with an edge every raw occupation, and
therefore every inclusion rank, lies in \(1,\ldots,M-1\).
For \(M=1\) the feasible active graph has no edges.

Relabeling \(S_v=\gamma_v S'_v\) changes maps to
\(\alpha'_{uv}=\gamma_v^{-1}\alpha_{uv}\gamma_u\) and bijects counted
families. Path transport is ordered composition in the direction of
travel. Reversal inverts a closed transport; changing its base point
conjugates it, preserving its number of fixed labels.

\begin{lemma}[Flatness]\label{lem:flat}
All closed-walk transports are identity if and only if all simple-cycle
transports are identity, if and only if vertex gauges make every edge
identity. A flat active cover has zero deficit.
Counts factor over connected components.
\end{lemma}
\begin{proof}
Choose a rooted spanning tree and let \(\gamma_v\) be transport from its
root to \(v\). Tree edges become identity; every chord map is the transport
around its fundamental simple cycle. Thus simple-cycle flatness makes
every map identity, which implies closed-walk flatness. Undoing a gauge
only conjugates closed transports. The counting claims follow from the
gauge bijection and independent choices in separate components.
\end{proof}
In a tree gauge the group of transports of closed walks based at its
root equals the group generated by chord maps: a walk is a word in
those maps, and each chord map is realized by a fundamental closed walk.
Changing the tree only changes the coordinate gauge, and changing the
root conjugates this group by path transport. This proves the stated
conjugacy invariance of the holonomy image and of \(\kappa\).

\begin{remark}[The double-cover case]\label{rem:two}
Let \(M=2\), let \(\occ\) be feasible on any bipartite base, and put
\(H=G[\{v:a_v=1\}]\). Count isolated vertices in \(c(H)\), and use
\(c(\varnothing)=0\). Then
\[
 \Pa_G(\occ)=2^{c(H)},\qquad
 \Za_{\widetilde G}(\occ)=
 \begin{cases}
 2^{c(H)},&\widetilde G|_{G_\occ}\text{ is flat},\\
 0,&\text{otherwise}.
 \end{cases}
\]
Indeed, an occupation-two vertex has only one selected set and is
isolated in the active graph by feasibility. On each component of \(H\),
the inclusion ranks are all one, so fixing one label determines every
other set along a spanning tree. Flatness gives two choices; a nonflat
cycle has the nonidentity element of \(S_2\), which fixes neither label,
and hence permits no choice. The same argument applies to an isolated
vertex of \(H\).
In particular, on a connected base with an edge and a positive feasible
profile, \(\occ=\one\). If the cover is nonflat, then
\(D_{\widetilde G}(\one)=2\), \(R_G(\one)=1\), and \(D=2R\),
regardless of treewidth or chordal bipartiteness.
This componentwise argument is the elementary two-label mechanism
underlying the projection-fibre comparison discussed in
\cite[proof of Theorem~2.1]{Csikvari}; we do not count it as a separate principal
contribution.
Under \(S_2\cong\{+1,-1\}\), flatness says that the product of edge
signs around every cycle is positive, which is the usual signed-graph
balance condition. For a related connection between dual DP counts and
balanced chromatic polynomials, see Mudrock and Sharbel
\cite{DualPolynomial}. Our two-label count follows directly from
inclusion coordinates and does not require a signed-coloring formula.
\end{remark}

\begin{proposition}[First homogeneous defect]\label{prop:first}
For a nonflat cover of any bipartite \(G\), let \(g\) be the minimum
length of a nonflat simple cycle. For an unoriented such cycle counted
once, put \(m_C=M-\#\Fix(\hol(C))\). Then
\[
 (p_G^M-\Pi_{\widetilde G})_{[d]}=0\ (d<g),\qquad
 (p_G^M-\Pi_{\widetilde G})_{[g]}
 =\sum_{\substack{|C|=g\\\hol(C)\ne1}}m_C\prod_{v\in V(C)}x_v.
\]
Each displayed vertex set supports a unique chordless cycle, so its
monomial is not repeated. The value \(m_C\) is independent of a chosen
base point and direction: conjugation and inversion preserve the number
of fixed labels. In particular, for every bipartite base,
\(p_G^M=\Pi_{\widetilde G}\) if and only if the cover is flat.
\end{proposition}
\begin{proof}
A support with fewer than \(g\) vertices has only flat simple cycles
and hence zero deficit. At total degree \(g\), a nonflat support must
contain a nonflat cycle of length \(g\), consist exactly of its vertices,
and have occupation one everywhere.
A chord would split this cycle into two shorter cycles. Orient the
two arcs between the chord endpoints in the same direction. If both
shorter cycles were flat, both arc transports would equal the chord
transport, making the original cycle flat. Thus the cycle is chordless
and is the entire induced graph on its support.

Gauge a spanning path of this cycle to identity, leaving one map \(\sigma\).
Let \(\mathsf A_M=J-I\) on \(\mathbb R^M\), where \(J\) is all ones.
The singleton gained count is \(\tr(\mathsf A_M^g P_\sigma)\), where \(P_\sigma\)
is the permutation matrix (using its inverse gives the same trace).
For \(E_0=J/M\),
\[
 \mathsf A_M^g=(M-1)^gE_0+(-1)^g(I-E_0),\quad
 \tr(E_0P_\sigma)=1,\quad
 \tr(P_\sigma)=\#\Fix(\sigma).
\]
Subtracting the gained count from the ordinary one gives
\[
 (-1)^g(M-\#\Fix(\sigma))=m_C>0,
\]
since \(g\) is even.
All other degree-\(g\) profiles are flat on their support or infeasible.
This proves the formula and, together with Lemma~\ref{lem:flat},
the polynomial equality assertion.
\end{proof}
This proposition does not assert nonnegativity of higher-degree
coefficients on an arbitrary bipartite base.

\begin{example}[One twisted square]\label{ex:square}
Let \(G=C_4=(v_0,v_1,v_2,v_3)\), with left side \(\{v_0,v_2\}\),
and take \(M=3\). Give the path \(v_0v_1v_2v_3\) identity transports
and the closing edge the transposition \(\tau=(12)\). The holonomy
image is \(\Gamma=\langle\tau\rangle\); it moves two labels, and
\(\kappa(\Gamma)=2\). Here are two profiles on this same cover:
\[
 \begin{array}{c|rrrr}
 (a_{v_0},a_{v_1},a_{v_2},a_{v_3})&P&Z&D&R\\\hline
 (1,1,1,1)&18&16&2&1\\
 (1,2,1,2)&3&1&2&1
 \end{array}
\]
For the singleton profile, fix the color at \(v_0\). There are eight
proper colorings of the identity path. Two end in the same color as
\(v_0\), and three end in each of the other two colors. The identity
closing edge therefore leaves six choices for each initial color,
giving \(P=18\). The twisted edge leaves five choices when the initial
color is 1 or 2, and six when it is 3. Thus \(Z=5+5+6=16\).

For \((1,2,1,2)\), disjointness on the path forces the two left sets
to be the same singleton and the two right sets to be its complement.
All three singletons work in the ordinary model, whereas the closing
edge requires \(\tau\) to fix that singleton. Only \(\{3\}\) works,
so \(P=3\) and \(Z=1\).
The reduced profiles on one label are respectively \((0,0,0,0)\) and
\((0,1,0,1)\); each has exactly one ordinary assignment, so \(R=1\).
Both coefficients attain \(D=\kappa R\). The second profile also shows
that the comparison concerns multiple selections in a fibre.
\end{example}

\section{Overlap kernels and elimination}\label{sec:networks}

For \(M\ge0\) and \(0\le p\le M\), the \emph{slice} \(\Om_p(M)\)
is the family of all \(p\)-subsets of \([M]\). In particular,
\([0]=\varnothing\) and \(\Om_0(0)=\{\varnothing\}\).
This convention makes the reduced model at \(M=2\) an ordinary
one-assignment model on zero labels. Set
\[
 \begin{aligned}
 \ell_{pq}&=(p+q-M)_+,\\
 r_{pq}&=\min(p,q)-\ell_{pq}=\min(p,q,M-p,M-q),\\
 z(A,B)&=|A\cap B|-\ell_{pq}.
 \end{aligned}
\]
All integers \(0,\ldots,r_{pq}\) occur as overlaps.
For \(r\ge0\), define the nonnegative Newton cone
\[
 \NN_r=\left\{f(z)=\sum_{d=0}^r c_d\binom zd:\ c_d\ge0\right\}.
\]
Its coefficients are \(\Delta^d f(0)\), where
\(\Delta f(z)=f(z+1)-f(z)\).

\begin{lemma}[Cone operations]\label{lem:cone}
The cones are closed under nonnegative sums, products, forward
differences, Pascal shifts, and integration with a nonnegative initial
value. Differences and shifts of a function on \(0,\ldots,r\) have
domain \(0,\ldots,r-1\); for \(r=0\) that domain is empty.
If \(f\in\NN_r\) is nonconstant, then \(r\ge1\),
\(\Delta f(r-1)>0\), and \(f(r)>0\).
\end{lemma}
\begin{proof}
Nonnegative sums preserve the coefficients. Assume \(r\ge1\). If
\(f(z)=\sum_{d=0}^r c_d\binom zd\), Pascal's identity gives, on
\(0\le z\le r-1\),
\[
 \Delta f(z)=\sum_{d=0}^{r-1}c_{d+1}\binom zd,\qquad
 f(z+1)=\sum_{d=0}^{r-1}(c_d+c_{d+1})\binom zd.
\]
The omitted term \(c_r\binom zr\) vanishes on this shorter domain.
Conversely, if \(u\in\NN_{r-1}\) has coefficients \(b_d\ge0\) and
\(a\ge0\), its discrete integral is
\[
 a+\sum_{j=0}^{z-1}u(j)=a+\sum_{d=0}^{r-1}b_d\binom z{d+1}\in\NN_r.
\]
For products, count ordered pairs of subsets of sizes \(d,e\) according
to the size \(j\) of their union. This gives
\[
 \binom zd\binom ze
 =\sum_{j=\max(d,e)}^{d+e}
 \frac{j!}{(j-d)!(j-e)!(d+e-j)!}\binom zj.
\]
Each coefficient counts the ordered pair of subsets by its union,
and is nonnegative. Terms with \(j>r\) vanish on \(0,\ldots,r\), so
products lie in the same cone. If \(f\) is nonconstant, some
\(c_d>0\) with \(1\le d\le r\). Consequently
\[
 \Delta f(r-1)=\sum_{d=1}^r c_d\binom{r-1}{d-1}>0,
 \qquad f(r)=\sum_{d=0}^r c_d\binom rd>0.
\]
For \(r=0\) every function in the cone is a nonnegative constant;
there is no difference or shift value to evaluate.
\end{proof}

For radial functions on the appropriate overlap domains, write
\[
 (f*_M g)(A,B)=
 \sum_{T\in\Om_m(M)}f(z(A,T))g(z(T,B)).
\]
Simultaneous label permutations are transitive on pairs with a fixed
intersection size. Hence this sum is a function \(h_M(z(A,B))\),
without any assumption about its Newton coefficients.

\begin{lemma}[Two-label difference identity]\label{lem:serialdiff}
Let \(1\le p,m,q\le M-1\) and \(h_M=f*_M g\). On \(M-2\) labels
with every rank reduced by one,
\begin{equation}\label{eq:serialdiff}
 \Delta h_M(z(A',B'))=
 \sum_{T'\in\Om_{m-1}(M-2)}
 \Delta f(z(A',T'))\Delta g(z(T',B')).
\end{equation}
This identity holds for arbitrary real-valued radial inputs.
\end{lemma}
\begin{proof}
Removing two labels and reducing each rank by one preserves each
minimum-overlap shift and reduces each radius by one. For example,
\[
 \ell^{(M-2)}_{p-1,q-1}=\ell^{(M)}_{pq},\qquad
 r^{(M-2)}_{p-1,q-1}=r^{(M)}_{pq}-1.
\]
Name the removed labels \(i,j\). Put \(A=A'\cup\{i\}\) and compare
\(B'\cup\{i\}\) with \(B'\cup\{j\}\) in the original convolution.
Their shifted overlaps with \(A\) are \(z(A',B')+1\) and \(z(A',B')\).
Middle sets containing both or neither removed label cancel.
The pair \(T'\cup\{i\},T'\cup\{j\}\) contributes
\[
 f(x+1)g(y+1)+f(x)g(y)-f(x+1)g(y)-f(x)g(y+1)
 =\Delta f(x)\Delta g(y),
\]
where \(x,y\) are the reduced overlaps. Every reduced output overlap
occurs, including the unique one when \(M-2=0\).
Each middle-set pair is counted once, so no slice-cardinality or
probability factor appears. The argument uses no serial-closure premise.
\end{proof}

\begin{lemma}[Serial closure]\label{lem:serial}
For \(M\ge0\), \(0\le p,m,q\le M\),
\(f\in\NN_{r_{pm}}\), and \(g\in\NN_{r_{mq}}\),
the convolution \(h_M=f*_M g\) belongs to \(\NN_{r_{pq}}\).
This includes all endpoint and middle ranks, including zero and \(M\).
If all three ranks are interior and both inputs are nonconstant,
then \(h_M\) is nonconstant.
\end{lemma}
\begin{proof}
Induct strongly on \(M\). If an endpoint rank is zero or \(M\),
the output overlap domain is a singleton, so radiality and nonnegativity
prove the assertion. If \(m\) is zero or \(M\), its slice has one set,
and both input overlaps are fixed; the convolution is a nonnegative
constant on the output domain. These cases include \(M=0,1\).

Otherwise all three ranks are interior and \(r=r_{pq}\ge1\).
Lemma~\ref{lem:cone} puts \(\Delta f,\Delta g\) in the corresponding
cones on \(M-2\) labels. The induction hypothesis and
Lemma~\ref{lem:serialdiff} give
\[
 \Delta h_M=(\Delta f)*_{M-2}(\Delta g)
 =\sum_{d=0}^{r-1}c_d\binom zd,\qquad c_d\ge0.
\]
The defining finite sum gives \(h_M(0)\ge0\). Discrete integration yields
\[
 h_M(z)=h_M(0)+\sum_{d=0}^{r-1}c_d\binom z{d+1}\in\NN_r.
\]
Thus the same induction covers every overlap-shift regime.
For nonconstant inputs choose nested reduced prefixes of sizes
\(p-1,m-1,q-1\), namely \(A'=[p-1]\), \(T'=[m-1]\),
and \(B'=[q-1]\), in one ordering of the \(M-2\) labels.
Their shifted overlaps are the last values of their respective domains.
Thus the summand at \(T'\) in \eqref{eq:serialdiff} is positive by
Lemma~\ref{lem:cone}, while all other summands are nonnegative.
Equation~\eqref{eq:serialdiff} therefore makes the last output
difference positive, also when \(M-2=0\).
\end{proof}

\begin{lemma}[Diagonal-exact split]\label{lem:split}
For \(f,g\in\NN_r\), define \(F,G\) by
\[
 \begin{aligned}
 F(0)&=f(0)g(0),&G(0)&=0,\\
 \Delta F(k)&=\Delta f(k)g(k+1),&
 \Delta G(k)&=f(k)\Delta g(k).
 \end{aligned}
\]
Then \(F,G\in\NN_r\) and
\begin{equation}\label{eq:split}
 f(x)g(y)\le F(x)+G(y),\quad F(z)+G(z)=f(z)g(z),\quad
 \Delta F(z)\ge\Delta f(z)\Delta g(z).
\end{equation}
If both inputs are nonconstant, \(F\) is nonconstant.
\end{lemma}
\begin{proof}
Assume first \(r\ge1\). By Lemma~\ref{lem:cone}, the functions
\(\Delta f\), \(\Delta g\), \(g(z+1)\), and the restriction of
\(f\) all lie in \(\NN_{r-1}\). Product closure puts \(\Delta F\)
and \(\Delta G\) in that cone. Their nonnegative initial values and
integration then put \(F,G\) in \(\NN_r\).
The discrete product rule is
\[
 \Delta(fg)(k)=\Delta f(k)g(k+1)+f(k)\Delta g(k)
              =\Delta F(k)+\Delta G(k).
\]
Since \(F(0)+G(0)=f(0)g(0)\), summing these differences proves
\(F(z)+G(z)=f(z)g(z)\) at every point.

To compare different arguments, subtract \(f(x)g(y)\) and use this
diagonal identity at \(y\) when \(x\ge y\), or at \(x\) when
\(x\le y\). The resulting slack is
\[
 F(x)+G(y)-f(x)g(y)=
 \begin{cases}
 \displaystyle\sum_{k=y}^{x-1}\Delta f(k)(g(k+1)-g(y)),&x\ge y,\\[3pt]
 \displaystyle\sum_{k=x}^{y-1}(f(k)-f(x))\Delta g(k),&x\le y.
 \end{cases}
\]
Every summand is nonnegative, because cone functions are nondecreasing.
Moreover \(g(z+1)=g(z)+\Delta g(z)\ge\Delta g(z)\), giving the
difference bound. If both inputs are nonconstant,
\(\Delta F(r-1)=\Delta f(r-1)g(r)>0\).
At \(r=0\), the only arguments are \(x=y=0\), and the asserted
identities follow from the initial values. These arguments include zero
functions and never divide by a kernel value.
\end{proof}

\subsection{Gain-decorated kernel lists}
A network consists of a finite simple graph \(H\), a slice
\(\Om_{p_v}(M)\) at each vertex, a scalar \(s\ge0\), and a finite list
\(\mathcal L_{uv}=((f_i,\sigma_i))_i\) at every oriented edge \(uv\).
Here \(f_i\in\NN_{r_{p_up_v}}\), and \(\sigma_i\) maps labels at \(v\)
to labels at \(u\). Choose one orientation of each undirected edge for
all products below. Define
\[
 K_{uv}(A,B)=\sum_i f_i(z(A,\sigma_iB)),\qquad
 \overline K_{uv}(A,B)=\sum_i f_i(z(A,B)).
\]
Reversing an edge inverts each map and transposes the kernel:
\(z(A,\sigma B)=z(B,\sigma^{-1}A)\). Thus neither network value
depends on the chosen orientations. An empty list is zero.
The list is part of the data: the companion \(\overline K\) need not
be determined by the evaluated gained matrix \(K\).
The two network values are
\[
 \mathcal Z=s\sum_{(A_v)}\prod_{uv}K_{uv}(A_u,A_v),\qquad
 \mathcal P=s\sum_{(A_v)}\prod_{uv}\overline K_{uv}(A_u,A_v).
\]

A tree decomposition of a graph is a tree of bags covering every vertex
and edge, such that the bags containing a vertex form a connected subtree.
Its width is the largest bag size minus one; treewidth is the minimum width.

\begin{lemma}[Low-degree elimination]\label{lem:lowdegree}
Every nonempty graph of treewidth at most \(w\) has a vertex of degree
at most \(w\). Deletion and edge contraction preserve the width bound.
In particular, smoothing a degree-two vertex preserves treewidth at most two.
\end{lemma}
\begin{proof}
Repeatedly delete a leaf bag contained in its neighboring bag.
In a remaining leaf bag choose a vertex absent from the neighbor; all
its neighbors lie in that bag. If only one bag remains, choose any vertex.
For contraction replace the two endpoints in every bag by the merged
vertex. Their occurrence subtrees meet in a bag covering the contracted
edge, so the merged occurrence set is connected.
Smoothing is contraction followed by deletion of any duplicate edge.
\end{proof}
This last assertion concerns the underlying simple graph. In a kernel
network, the old and smoothed edge factors are both retained in the
weight; the split below handles their product rather than discarding
either constraint.

\begin{proposition}[Network comparison]\label{prop:network}
Every gain-decorated kernel-list network of treewidth at most two
satisfies \(\mathcal Z\le\mathcal P\).
\end{proposition}
\begin{proof}
Let the initial values be \(Z_0,P_0\). Eliminate vertices of degree at
most two while preserving
\begin{equation}\label{eq:invariants}
 Z_0\le\mathcal Z_j,\qquad P_0=\mathcal P_j.
\end{equation}
An isolated vertex is removed with scalar factor \(\binom M{p_v}\),
exactly in both models. A leaf is removed with its row sum. Each list
term has a row sum independent of the surviving set by slice transitivity,
and independent of its map by relabeling the summed set. Thus gained
and companion row sums agree, including when they vanish.

At a degree-two vertex with neighbors \(u,w\), orient maps along
\(u\to v\to w\) in the kernel convention. Expand the two lists and sum
the middle set. For maps \(\sigma,\tau\), substituting \(T=\sigma A_v\)
gives explicitly
\[
 \begin{aligned}
 &\sum_{A_v\in\Om_{p_v}(M)}
 f(z(A_u,\sigma A_v))g(z(A_v,\tau A_w))\\
 &\qquad=\sum_{T\in\Om_{p_v}(M)}
 f(z(A_u,T))g(z(T,\sigma\tau A_w))
 =h(z(A_u,\sigma\tau A_w)).
 \end{aligned}
\]
Lemma~\ref{lem:serial} puts this \(h\) in
\(\NN_{r_{p_up_w}}\). Its ordinary companion is the same convolution
with identity map. Expanding all entry pairs therefore produces a new
list of valid cone kernels, exactly in both models.
If \(uw\) already exists, expand the product of its list and the new
serial list. For each entry pair \((h,\sigma),(g,\tau)\),
Lemma~\ref{lem:split} replaces the product by the list entries
\((F,\sigma),(G,\tau)\). The gained kernel increases pointwise, while
the companion is unchanged because \(F+G=hg\) on the diagonal.
Indeed, for fixed endpoint sets put \(x=z(A_u,\sigma A_w)\) and
\(y=z(A_u,\tau A_w)\). The gained product is \(h(x)g(y)\), bounded
by \(F(x)+G(y)\). In the companion both arguments are \(z(A_u,A_w)\),
where the replacement is exact. Summing over all entry pairs and
endpoint assignments therefore gives
\(\mathcal Z_{j+1}\ge\mathcal Z_j\) and
\(\mathcal P_{j+1}=\mathcal P_j\).
The new list has at most twice the product of the old list lengths;
parallel contributions are stored on one simple edge.
All other factors are nonnegative, so \eqref{eq:invariants} persists.

Each step removes a vertex and preserves width by
Lemma~\ref{lem:lowdegree}. At termination both values equal the final
scalar. This proves the comparison without positivity or division
assumptions.
\end{proof}

For a feasible bipartite edge, \(p_u\le p_v\) in inclusion coordinates,
and its constraint is the cone factor
\begin{equation}\label{eq:indicator}
 \ind{\alpha_{uv}A_u\subseteq A_v}
 =\binom{z(A_u,\alpha_{uv}^{-1}A_v)}{r_{p_up_v}}.
\end{equation}
Indeed maximum overlap means inclusion; the highest binomial basis is
one there and zero elsewhere. Proposition~\ref{prop:network} proves the
width-two comparison in Theorem~\ref{thm:A}.

\section{Sharp deficits on treewidth two}\label{sec:width}

The essential terminal case is a collision of two different transports,
even when each reduced branch is flat.
After smoothing a vertex, a serial kernel \((h,\sigma)\) and an old
edge kernel \((g,\tau)\) are multiplied at the same endpoints.
If retaining either transport makes the remaining graph flat but the
original is nonflat, then \(\sigma\ne\tau\) and this edge is a bridge.
Deleting it leaves two flat components:
\[
 \underbrace{\text{flat component rooted at }u}_{\text{constant }c_u}
 \quad
 \mathop{\longleftrightarrow}^{h,\sigma}_{g,\tau}
 \quad
 \underbrace{\text{flat component rooted at }w}_{\text{constant }c_w}.
\]
The two displayed kernels are multiplied, not added.
Their relative transport gives the local deficit in the following lemma.

The spectral input is a common nonnegative decreasing scalar sequence
for all cone kernels. Its monotonicity produces nonnegative weights for
cumulative moved-subset counts.

\begin{proposition}[Common decreasing slice scalars]\label{prop:scalars}
Fix interior ranks \(p,q\), put \(r=r_{pq}\), and give
\(V_k=\mathbb R^{\Om_k(M)}\) its counting inner product.
There are permutation-invariant spaces \(\mathcal H_i\) of dimension
\(\binom Mi-\binom M{i-1}\), and equivariant isometries
\(J_i^{(p)}:\mathcal H_i\to V_p\),
\(J_i^{(q)}:\mathcal H_i\to V_q\), with orthogonal ranges for distinct
\(i\), such that every \(f\in\NN_r\) has kernel
\[
 K_f(A,B)=f(z(A,B)),\qquad
 K_f=\sum_{i=0}^r k_i(f)J_i^{(p)}(J_i^{(q)})^*,
\]
where
\[
 k_0(f)\ge k_1(f)\ge\cdots\ge k_r(f)\ge0.
\]
The isometries are common to all \(f\) for the fixed pair of ranks.
\end{proposition}
\begin{proof}
Appendix~\ref{app:spectral} constructs the harmonic spaces and
counting-normalized lifts, derives the scalars of the binomial basis
kernels, and proves their monotonicity. Complementing both endpoint sets
handles a positive overlap shift without a normalization factor.
\end{proof}

\begin{lemma}[Local holonomy loss]\label{lem:local}
Let \(1\le p,q\le M-1\), \(f,g\in\NN_{r_{pq}}\), and
\(\gamma\in S_M\). Define, with sums over the indicated slices,
\[
 S_\gamma=\sum_{A\in\Om_p(M),\,B\in\Om_q(M)}
 f(z(A,B))\bigl(g(z(A,B))-g(z(\gamma A,B))\bigr),
\]
\[
 Q=\sum_{A'\in\Om_{p-1}(M-2),\,B'\in\Om_{q-1}(M-2)}
 \Delta f(z(A',B'))\Delta g(z(A',B')).
\]
For every transposition \(t\), \(S_t=2Q\).
If \(\Gamma\le S_M\) is nontrivial and \(\gamma\in\Gamma\setminus\{1\}\),
then \(S_\gamma\ge\kappa(\Gamma)Q\).
Furthermore \(\kappa(\Gamma)\ge2\), and odd order implies
\(\kappa(\Gamma)\ge3\).
\end{lemma}
\begin{proof}
Removing two labels and reducing both ranks by one preserves the overlap
shift and reduces the radius by one. For \(t=(i\,j)\), symmetrizing in \(A,tA\)
gives
\[
 S_t=\frac12\sum_{A,B}
 (f(z(A,B))-f(z(tA,B)))
 (g(z(A,B))-g(z(tA,B))).
\]
Only sets \(A,B\) containing exactly one of \(i,j\) contribute.
Each reduced pair has four such membership patterns, with the same
product of forward differences; the factor \(1/2\) leaves \(2Q\).
Both/neither patterns vanish, including at reduced boundary ranks.

Proposition~\ref{prop:scalars} yields nonnegative nonincreasing scalars
\(k_i(f),k_i(g)\). With
\(\mu_i=k_i(f)k_i(g)\), \(r=r_{pq}\), and \(\mu_{r+1}=0\), it gives
\begin{equation}\label{eq:localweights}
 S_\gamma=\sum_{k=1}^r(\mu_k-\mu_{k+1})m_k(\gamma).
\end{equation}
Every weight is nonnegative. If \(\gamma(i)=j\ne i\), a fixed subset
must have the same membership at \(i,j\), so
\(m_k(\gamma)\ge2\binom{M-2}{k-1}\).
If \(\gamma\) has a cycle of length at least three, choose three labels
on it. A fixed subset has constant membership on them, and the six
nonconstant patterns give
\[
 m_k(\gamma)\ge
 3\binom{M-3}{k-1}+3\binom{M-3}{k-2}
 =3\binom{M-2}{k-1}.
\]
A transposition has exactly the first number of moved subsets.
Put \(w_k=\mu_k-\mu_{k+1}\ge0\). Substituting a transposition in
\eqref{eq:localweights} and using \(S_t=2Q\) gives
\[
 Q=\sum_{k=1}^r w_k\binom{M-2}{k-1}.
\]
Since \(r\le\lfloor M/2\rfloor\), the defining minimum
\eqref{eq:kappa} applies to every term. Hence
\[
 S_\gamma=\sum_{k=1}^r w_k m_k(\gamma)
 \ge\kappa(\Gamma)\sum_{k=1}^r w_k\binom{M-2}{k-1}
 =\kappa(\Gamma)Q.
\]
The two moved-subset estimates prove the uniform constants.
In an odd-order group every nonidentity element has a cycle of length
at least three, proving the odd-order assertion.
Finally \(m_k(\gamma^{-1})=m_k(\gamma)\), because the fixed subsets
are the same. Equation~\eqref{eq:localweights} therefore also proves
\(S_{\gamma^{-1}}=S_\gamma\); equivalently the real orthogonal
representations in Appendix~\ref{app:spectral} have equal traces on a
permutation and its inverse.
\end{proof}

\begin{lemma}[Ordinary insertion]\label{lem:insertion}
In an ordinary overlap network with interior original ranks, the original
sum with factors \(f_e\) is at least the sum on \(M-2\) labels, ranks
reduced by one, and factors \(\Delta f_e\).
For one prescribed original root set, the respective rooted constants
satisfy \(c\ge c^-\). This includes the leaf row sums.
\end{lemma}
\begin{proof}
Add a fixed new label \(i\) to every reduced set and omit a second
new label \(j\). Every overlap argument increases by one, and
\(f(z+1)\ge\Delta f(z)\), so the injection does not decrease weight.
For a prescribed root \(A\), choose \(i\in A\), \(j\notin A\),
and prescribe \(A\setminus\{i\}\) on the remaining labels.
Ordinary sums at a root are constant over its slice by simultaneous
permutations. The same injection proves the rooted assertion.
\end{proof}
This lemma concerns ordinary or gauged-flat pieces. It does not assert
insertion into arbitrary gained fibres.

\begin{theorem}[Quantitative overlap network]\label{thm:aux}
Let \(H\) be connected and simple, have an edge and treewidth at most two,
and have interior ranks on \(M\ge2\) labels.
Put one nonconstant cone factor and one transport in a fixed
nontrivial group \(\Gamma\le S_M\) on each edge.
Let \(P_f,Z_f\) be its ordinary and gained sums, and let \(R_f\) be its
ordinary sum after reducing all ranks by one, removing two labels,
and replacing every factor by its forward difference.
If the transport system is nonflat, then
\[
 P_f-Z_f\ge\kappa(\Gamma)R_f,\qquad R_f>0.
\]
\end{theorem}
\begin{proof}
Nested prefixes give the positive last difference of every factor,
so \(R_f>0\).
Induct on the number of vertices, retaining all the stated hypotheses:
connectedness, simplicity, width at most two, an edge, interior ranks,
one nonconstant factor and one map in \(\Gamma\) per edge, and nonflatness.
The ambient \(M\) and surviving ranks do not change during this induction.
Choose a vertex of degree at most two by Lemma~\ref{lem:lowdegree}.
It cannot be isolated.

For a leaf, let \(L\) be the common row sum of its edge kernel when
summing over the leaf slice, with the neighbor set fixed. Let \(L^-\)
be the analogous row sum after removing two labels, reducing both ranks
by one, and replacing the factor by its forward difference. These row
sums are independent of the fixed neighbor set and of the transport.
Then \(P_f-Z_f=L(P_{\rm rest}-Z_{\rm rest})\) and
\(R_f=L^-R_{\rm rest}\). Removing it preserves a nonflat cycle, hence
an edge and all inductive hypotheses. Lemma~\ref{lem:insertion}
gives \(L\ge L^-\), so induction applies.
A one-vertex or one-edge nonflat instance is impossible.

For degree two with no existing neighbor edge, serial elimination is
exact and creates a nonconstant factor by Lemma~\ref{lem:serial}.
The transport is a product of maps in \(\Gamma\); smoothing preserves
width, connectedness, and holonomies. The reserve is unchanged by
\eqref{eq:serialdiff}, so induction applies on one fewer vertex.

If the neighbor edge already exists, let \((h,\sigma)\) be the serial
factor and map, and \((g,\tau)\) the old edge.
Consider the two reduced transport systems retaining just \(\sigma\)
or just \(\tau\) there. If one is nonflat, choose it as the \(F\) branch
of Lemma~\ref{lem:split}, interchanging the inputs if necessary.
That branch has a nonconstant \(F\) and all inductive hypotheses.
The \(G\) branch may have a constant or zero factor, but its deficit is
nonnegative by Proposition~\ref{prop:network}.
For clarity, let \(P_F,Z_F,P_G,Z_G\) be the full sums of the two
branches, each with all the other factors retained. Their defining
identities and inequalities are
\[
 P_f=P_F+P_G,\qquad Z_f\le Z_F+Z_G,\qquad
 D_f\ge D_F+D_G.
\]
In the reserve, the serial factor is \(\Delta h\) by
\eqref{eq:serialdiff}, and the old factor is \(\Delta g\).
The bound \(\Delta F\ge\Delta h\,\Delta g\), multiplied by the
remaining nonnegative reduced factors and summed, gives \(R_F\ge R_f\).
Induction on the nonflat \(F\) branch now yields
\[
 P_f-Z_f\ge D_F+D_G\ge D_F
 \ge\kappa(\Gamma)R_F\ge\kappa(\Gamma)R_f.
\]
The last comparison uses \eqref{eq:serialdiff} and nonnegative
remaining reduced factors.

It remains that both reduced transport systems are flat.
Then \(\sigma\ne\tau\); otherwise a reduced flat gauge would extend
across the eliminated vertex. Their common edge is a bridge: an
alternative path between its endpoints would force both maps to equal
that path transport. Delete the bridge and gauge its two flat components
to identity. The component gauges can be chosen in \(\Gamma\).
Their fixed-root sums are constants \(c_u,c_w>0\): a permutation of
nested prefixes realizes any prescribed root set with positive weight.
Their reduced constants satisfy \(c_u\ge c_u^-\), \(c_w\ge c_w^-\)
by Lemma~\ref{lem:insertion}.
For a component consisting only of its root use \(c=c^-=1\).

After these component gauges, denote the two bridge maps again by
\(\sigma,\tau\). Relabel every vertex of the \(w\) component by the
same permutation: write each old set as \(\sigma^{-1}\) times its new
set. All its internal identity maps remain identity, and the two bridge
maps become \(1\) and \(\gamma=\tau\sigma^{-1}\ne1\), with
\(\gamma\in\Gamma\).
Summing the component interiors first now gives
\[
 \begin{aligned}
 P_f-Z_f
 &=c_uc_w\sum_{A,B}h(z(A,B))
       \bigl(g(z(A,B))-g(z(A,\gamma B))\bigr)\\
 &=c_uc_w S_{\gamma^{-1}}=c_uc_w S_\gamma.
 \end{aligned}
\]
Here \(z(A,\gamma B)=z(\gamma^{-1}A,B)\), and the last equality
was proved in Lemma~\ref{lem:local}. The reduced component sums and
the two-label serial identity similarly give
\[
 R_f=c_u^-c_w^-Q.
\]
Lemma~\ref{lem:local} proves the bound. This case terminates the induction;
it does not invoke a nonflat induction hypothesis on a flat branch.
\end{proof}
The auxiliary graph may have triangles. This theorem is about nondecreasing
overlap kernels; it asserts no independent-set comparison on
nonbipartite bases.

\begin{proof}[Proof of Theorem~\ref{thm:B} and width-two rigidity]
A spanning-tree gauge puts all maps in the holonomy image \(\Gamma\).
At a positive feasible profile, \eqref{eq:indicator} is nonconstant
with \(r_{p_up_v}\ge1\). Its forward difference is
\(\binom z{r_{p_up_v}-1}\), precisely the inclusion indicator on the
reduced slices. Both raw occupations reduce by one, since on the right
\((M-2)-(p_v-1)=a_v-1\).
Theorem~\ref{thm:aux} proves the stated \(\kappa\) bound.
The reduced profile is feasible:
\((a_u-1)+(a_v-1)\le M-2\).
Lemma~\ref{lem:counting} gives its positive count, including the unique
empty assignment at \(M=2\).

For a general feasible width-two active graph remove isolated positive
factors. Every remaining nonflat connected component has a strictly
positive deficit by the bound just proved; every component has
\(0\le Z_i\le P_i\), \(P_i>0\).
Thus \(\prod Z_i=\prod P_i\) exactly when each component has equality,
which is exactly active flatness. The flat implication follows from
Lemma~\ref{lem:flat}. At \(M=1\) only isolated active vertices remain.
\end{proof}

\begin{corollary}[Profile-dependent holonomy bound]\label{cor:profilekappa}
Under the hypotheses of Theorem~\ref{thm:B}, put
\[
 r^*(\occ)=\max_{v\in V(G)}\min(a_v,M-a_v).
\]
For \(1\le s\le\lfloor M/2\rfloor\), define
\[
 \kappa_s(\Gamma)=
 \min_{\substack{\gamma\in\Gamma\setminus\{1\}\\1\le k\le s}}
 \frac{m_k(\gamma)}{\binom{M-2}{k-1}}.
\]
Then
\[
 \Da_{\widetilde G}(\occ)
 \ge\kappa_{r^*(\occ)}(\Gamma)\Ra_G(\occ).
\]
In particular, for a noncanonical full \(M\)-fold DP cover
\(\mathfrak H\) of such a graph, with holonomy image \(\Gamma\),
\[
 P(G,M)-P_{DP}(G,\mathfrak H)\ge\operatorname{mindeg}(\Gamma),
 \qquad
 \operatorname{mindeg}(\Gamma)
 =\min_{\gamma\ne1}\#\{i:\gamma(i)\ne i\}.
\]
\end{corollary}
\begin{proof}
The ranks at surviving vertices are unchanged throughout the induction
in Theorem~\ref{thm:aux}. Each local bridge estimate therefore uses a
radius
\[
 r_{p_up_w}\le
 \min\{\min(p_u,M-p_u),\min(p_w,M-p_w)\}\le r^*(\occ).
\]
The second inequality holds because complementing a right occupation
does not change its distance from the boundary ranks.
In Lemma~\ref{lem:local} the weights only run to that bridge radius.
The defining minimum for \(\kappa_{r^*(\occ)}\) thus bounds every
moved-subset term used there. Replace \(\kappa(\Gamma)\) by this single
constant in the entire quantitative induction. The leaf step uses
\(L\ge L^-\) with the same constant on the remaining graph; the serial
step preserves both counts and the reserve; the split step retains
\(R_F\ge R_f\); and the terminal step has the improved local bound.
This proves the first claim.
Notice that \(1\le r^*(\occ)\le\lfloor M/2\rfloor\) under the
positive feasible hypotheses.

At the singleton profile \(r^*=1\) and \(\Ra_G(\one)=1\).
Since \(m_1(\gamma)\) is exactly the number of moved labels and
\(\binom{M-2}{0}=1\), we have
\(\kappa_1(\Gamma)=\operatorname{mindeg}(\Gamma)\).
The DP identification proves the final assertion.
\end{proof}
The restricted minimum satisfies \(\kappa_s(\Gamma)\ge\kappa(\Gamma)\).
It improves selected coefficients; the unrestricted constant remains
the one used for a bound on the entire polynomial.

\begin{corollary}[Polynomial and activity bounds]\label{cor:polynomial}
Under the graph and cover hypotheses of Theorem~\ref{thm:B},
\[
 p_G^M-\Pi_{\widetilde G}\coefge
 \kappa(\Gamma)\xx^V p_G^{M-2}.
\]
For \(\lambda_v>0\), with \(Z_G=p_G(\boldsymbol\lambda)\), this implies
\[
 \frac{\Pi_{\widetilde G}(\boldsymbol\lambda)}{Z_G^M}
 \le1-\frac{\kappa(\Gamma)\prod_v\lambda_v}{Z_G^2}.
\]
\end{corollary}
\begin{proof}
If a coefficient on the right is nonzero, the reduced profile is
feasible. Since every vertex has a neighbor, the original profile is
positive, feasible, and interior; apply Theorem~\ref{thm:B}.
If the right coefficient vanishes, network comparison supplies
nonnegativity. Substitution of positive activities and division by
\(Z_G^M>0\) gives the second formula.
\end{proof}

\begin{example}[Sharp uniform constants]\label{ex:sharp}
For \(G=C_{2\ell}\), \(\ell\ge2\), put a transposition on one edge and
identity on the others. Its two moved labels form a nonflat two-cover
\(Y\); its \(M-2\) fixed labels give ordinary copies. Thus
\(\Pi_{\widetilde G}=\Pi_Yp_G^{M-2}\).
A proper support of the cycle is a forest, so has no two-cover defect.
On full support the only feasible two-label profile is all ones;
the ordinary count is two, while the transposition makes both
alternating selections inconsistent around the cycle. Hence
\[
 p_G^M-\Pi_{\widetilde G}=2\xx^V p_G^{M-2}.
\]
For one \(3\)-cycle edge, the singleton deficit on an even cycle is
three by the trace calculation in Proposition~\ref{prop:first},
whereas its reduced all-zero coefficient is one.
This proves sharpness of three for an odd-order holonomy image.
\end{example}

\begin{example}[An even-order group with sharp factor three]\label{ex:even}
Let \(M=6\), \(\theta=(12)(34)(56)\), and \(\Gamma=\{1,\theta\}\).
A fixed subset is a union of the three two-element orbits. Therefore
\[
\begin{array}{c|ccc}
 k&1&2&3\\ \hline
 \#\Fix_{\Om_k}(\theta)&0&3&0\\
 m_k(\theta)&6&12&20\\
 m_k(\theta)/\binom4{k-1}&6&3&10/3
\end{array}
\]
and \(\kappa(\Gamma)=3\), although \(|\Gamma|=2\).
On \(C_4\) use raw left occupations \((2,2)\), raw right occupations
\((4,4)\), and one \(\theta\)-edge. The three identity edges already force
one common left \(2\)-set \(A\), with every right set its complement.
The twisted constraint is \(\theta A=A\), with three choices.
The ordinary count is \(\binom62=15\); in the reduced model on four
labels one common left singleton gives reserve \(\binom41=4\).
Thus \(D=15-3=12=3R\).
For the singleton profile on the same cover,
\(P(C_4,6)=5^4+5=630\), while the cycle trace calculation gives
\(D=m_1(\theta)=6\) and hence \(Z=624\). Here \(R=1\) and
\(\kappa_1(\Gamma)=6\), so the profile-dependent bound is sharp too.
The saturated profile has \(r^*=2\) and \(\kappa_2(\Gamma)=3\);
the two profiles explain why the full-polynomial constant can be
smaller than the singleton constant.
\end{example}

\begin{example}[Mixed occupations and noncommuting holonomy]\label{ex:mixedranks}
Let \(G=K_{2,3}\) have left vertices \(u,w\) and right vertices
\(x_1,x_2,x_3\). Set \(M=3\),
\[
 (a_u,a_w;a_{x_1},a_{x_2},a_{x_3})=(1,1;1,2,1).
\]
Put identity maps on \(ux_i\), and maps \(1,t,s\) on \(wx_i\), where
\(t=(12)\) and \(s=(13)\). The identity edges, including \(wx_1\),
form a spanning tree; the two remaining maps generate \(S_3\), so
this holonomy is noncommutative and \(\kappa(S_3)=2\).
The bags \(\{u,w,x_i\}\) form a width-two decomposition.

Summing a right vertex gives the kernels \(1+z,z,1+z\), respectively,
on the singleton inclusion sets \(A\) at \(u\) and \(B\) at \(w\).
If \(\delta_i=\ind{A=\sigma_iB}\), with
\((\sigma_1,\sigma_2,\sigma_3)=(1,t,s)\), then
\[
 P=3\cdot2\cdot1\cdot2=12,\qquad
 Z=\sum_{B\in\Om_1(3)}(1+\ind{tB=B})(1+\ind{tB=sB})=4.
\]
The three summands in \(Z\) are \(1,1,2\).
The reduced profile on one label is \((0,0;0,1,0)\), hence \(R=1\)
and \(D=8\).

To see the split in the proof, smooth \(x_1\), then \(x_2\).
The factors \((z,t)\) and \((1+z,1)\) are parallel on \(uw\).
Lemma~\ref{lem:split}, with \(f=z\), \(g=1+z\), gives \(F=2z\)
and a zero second branch. The retained branch has
\[
 P_F=12,\quad
 Z_F=\sum_{B\in\Om_1(3)}2(1+\ind{tB=sB})=6,\quad R_F=2.
\]
Thus the gained envelope slack is \(6-4=2\), while
\(\Delta F=2\ge\Delta z\,\Delta(1+z)=1\) preserves the original
reserve \(R=1\). Smoothing \(x_3\) leaves two transports on one edge.
Their relative map \(t^{-1}s\) is a three-cycle, and its local loss is
\(6\), with reduced local product \(2\).
The full deficit is therefore \(8=2+6\), displaying both the split
slack and the terminal holonomy loss.
\end{example}

\begin{corollary}[Disconnected width-two graphs]\label{cor:disconnected}
For a positive feasible profile on a possibly disconnected width-two
bipartite graph, a nonflat component \(J\) gives
\(\Da_{\widetilde G}\ge\kappa(\Gamma_J)\Ra_G\) and
\(\Da_{\widetilde G}>0\). The reserve is positive if every isolated
occupation is at most \(M-1\).
\end{corollary}
\begin{proof}
Product expansion gives
\(D_G\ge D_J\prod_{i\ne J}P_i>0\).
If \(R_G>0\), all isolated ranks too are interior, and ordinary insertion
gives \(P_i\ge R_i\) for every other component. Multiply these bounds
with Theorem~\ref{thm:B}. If \(R_G=0\), scalar comparison suffices.
Reduced feasibility proves the stated positivity condition.
\end{proof}
\begin{remark}[Odd-order holonomy]
The odd-order hypothesis in these statements concerns the entire
holonomy image. Odd-order generators alone do not imply that the
group they generate has odd order.
\end{remark}

\section{Chordal-bipartite comparison and rigidity}\label{sec:chordal}

After zeros and common isolated factors are removed, the proof tracks
three nonnegative losses:
\begin{equation}\label{eq:three}
 D_{\widetilde G}(\occ)=\varepsilon+\Lambda D_Y+\Lambda(Z_Y-Z_H),\qquad \varepsilon\ge0,\ \Lambda>0.
\end{equation}
Here \(\varepsilon\) is the surplus from saturating one side, \(H\) is a smaller
contracted multigraph, and \(Y\) retains one labelled edge from each
parallel class. The terms represent saturation surplus, the smaller
graph deficit, and assignments rejected by omitted constraints.
We now construct the identity and prove each assertion about its terms.

\subsection{Compression and a nondecreasing product cone}

The operation in the proof is the usual two-label shifting operation
of extremal set theory; see Frankl's survey \cite{Frankl}.
Here it is applied to several prescribed sets at once and to a
product of arbitrary nonnegative nondecreasing intersection costs.
KMS use a rearrangement inequality to compare complete-bipartite
singleton counts \cite{KMS}. The compression below has the same
alignment principle, but retains the entire product cone needed for
successive nest-column eliminations and for saturation.

\begin{lemma}[Simultaneous compression]\label{lem:compression}
For sets \(D_i\subseteq[M]\) of sizes \(d_i\) and nonnegative
nondecreasing functions \(h_i\),
\[
 \sum_{Q\in\Om_q(M)}\prod_i h_i(|Q\cap D_i|)
 \le \sum_{Q\in\Om_q(M)}\prod_i h_i(|Q\cap[d_i]|).
\]
Equality holds when the \(D_i\) are nested.
\end{lemma}
\begin{proof}
For labels \(a<b\), replace \(b\) by \(a\) simultaneously in every
\(D_i\) containing \(b\) but not \(a\).
Partition \(q\)-sets into orbits under \((a\,b)\). Fixed orbits do not
change. In an orbit \(F\cup\{a\},F\cup\{b\}\), let \(I,J\) index the
sets containing only \(a\) and only \(b\), respectively. Write
\(l_i=h_i(|F\cap D_i|)\), \(r_i=h_i(|F\cap D_i|+1)\).
After suppressing common nonnegative factors, the increase is
\[
 \left(\prod_{i\in I}r_i-\prod_{i\in I}l_i\right)
 \left(\prod_{j\in J}r_j-\prod_{j\in J}l_j\right)\ge0.
\]
Empty products cause no problem, and no zero factor is divided out.
Each nontrivial compression decreases the sum of labels in all sets.
At termination each set is its prescribed prefix.
Nested sets can already be taken to those prefixes by one common
permutation, proving equality.
\end{proof}

Let \(\CC_s\) consist of finite sums of products
\(F(d_1,\ldots,d_s)=\sum_t c_t\prod_i f_{t,i}(d_i)\),
where \(c_t\ge0\) and each \(f_{t,i}\) is nonnegative and nondecreasing
on \(0,\ldots,M\).
For \(I\subseteq[s]\), define
\[
 (T_{I,q}F)(\mathbf d)=\sum_{Q\in\Om_q(M)}F(\mathbf y),\qquad
 y_i=\begin{cases}|Q\cap[d_i]|&i\in I,\\d_i&i\notin I.\end{cases}
\]
Compression bounds the analogous sum with arbitrary sets of sizes
\(d_i\), with equality for nested sets on \(I\).
Also \(T_{I,q}F\in\CC_s\): for every fixed \(Q\) and product generator,
each function \(d\mapsto f_{t,i}(|Q\cap[d]|)\) is nonnegative and
nondecreasing, so the finite sum remains in the cone.

For row neighborhoods \(J_i=N(i)\subseteq R\), a \emph{beta order}
eliminates at each step a column whose incident row neighborhoods,
restricted to the remaining columns, are totally ordered by inclusion.
Such a column is a \emph{nest point}. In standard terminology these
are beta elimination orders of a \emph{beta-acyclic hypergraph}.
Equivalently its incidence graph has no induced cycle of length at
least six; Appendix~\ref{app:structure} proves the equivalence with the
order convention used here. These notions are classical
\cite{Fagin,BrouwerKolen,Ausiello,BraultBaron}.

\begin{lemma}[Root-last order]\label{lem:rootlast}
A finite bipartite graph is chordal bipartite if and only if it has
a beta order. In that case a beta order may end at any prescribed column,
and the same assertion holds with the two sides transposed.
Isolated columns and empty or repeated row neighborhoods are allowed.
\end{lemma}
\begin{proof}
Appendix~\ref{app:structure} derives the order, including the
root-last refinement, from a self-contained two-nest-point argument.
\end{proof}

\begin{proposition}[Root-retaining cone comparison]\label{prop:conecompare}
Put \(Q_v=[M]\setminus S_v\) on the right and
\[
 E_i^\sigma(Q)=\bigcap_{v\in J_i}\sigma_{iv}Q_v,\qquad
 \sigma_{iv}=\alpha_{iv}^{-1},
\]
with empty intersection \([M]\).
Fix every \(q_v=|Q_v|\). If the column hypergraph has a beta order, then
for every \(F\in\CC_{|L|}\) its gained sum on
\((|E_i^\sigma|)_i\) is at most its identity sum.
The comparison also holds with one \(Q_\rho\) prescribed when the
order ends at \(\rho\).
\end{proposition}
\begin{proof}
We prove a statement for every residual family, rather than only for
the binomial cost of independent sets. Let \(W\) be the remaining
columns, let \(J_i\subseteq W\) be the residual neighborhoods, and fix
the ranks \(q_v\), all residual transports, and a beta order on \(W\).
If a root is retained, prescribe its set and require that this order
end at the root. Write
\[
 \Psi^\sigma_{\mathbf J}(F)=
 \sum_{(Q_v)}F\left(
 \left|\bigcap_{v\in J_1}\sigma_{1v}Q_v\right|,\ldots,
 \left|\bigcap_{v\in J_s}\sigma_{sv}Q_v\right|\right),
\]
where the sum omits the prescribed root. Our induction assertion is
\(\Psi^\sigma_{\mathbf J}(F)\le\Psi^\id_{\mathbf J}(F)\) for
\emph{every} \(F\in\CC_s\), including empty or repeated residual rows.
The induction parameter is the number of summed columns.
With none, each intersection is \([M]\), except that a row incident
with a retained root has intersection \(\sigma_{i\rho}Q_\rho\).
Its cardinality is \(q_\rho\), so the two values agree.

For the first column \(v\), hold the other sets fixed and put
\[
 D_i^\sigma=\bigcap_{w\in J_i\setminus\{v\}}\sigma_{iw}Q_w,
 \qquad I_v=\{i:v\in J_i\}.
\]
On \(I_v\) the cardinality is
\(|\sigma_{iv}^{-1}D_i^\sigma\cap Q_v|\).
Compression bounds the \(v\)-sum by
\((T_{I_v,q_v}F)((|D_i^\sigma|)_i)\).
Put \(F_v=T_{I_v,q_v}F\). The closure proved above gives
\(F_v\in\CC_s\), and the suffix is a beta order of the traced
neighborhoods \(J_i\setminus\{v\}\), still ending at the root.
The induction assertion therefore applies to this new cost and these
residual rows. In full-sum notation the two comparisons are
\[
 \Psi^\sigma_{\mathbf J}(F)
 \le\Psi^\sigma_{\mathbf J-v}(F_v)
 \le\Psi^\id_{\mathbf J-v}(F_v).
\]
At identity the \(D_i^\id\), \(i\in I_v\), are nested: incident
neighborhoods are nested, and intersection reverses containment.
The equality clause of compression therefore recovers the exact
ordinary \(v\)-sum at each residual configuration. Consequently
\(\Psi^\id_{\mathbf J-v}(F_v)=\Psi^\id_{\mathbf J}(F)\),
which completes the induction. This exact last equality is why the
compressed upper bound is the original ordinary count.
\end{proof}
The exact independent-set count is the sum with cost
\(\prod_i\binom{|E_i^\sigma|}{a_i}\), a cone generator.
Lemma~\ref{lem:rootlast} consequently proves chordal-bipartite coefficient
comparison and one-root comparison. Transposing the two sides retains
a left root. Empty and full ranks cause no exception;
\(T_{\varnothing,q}F=\binom MqF\).

\subsection{Saturation and tight-star contraction}
Assume now that the feasible active graph is nonempty and has no
isolated vertices, so all ranks are interior. Define
\[
 b_v=M-\max_{u\in N(v)}a_u,\quad
 \eta_v=\binom{b_v-1}{a_v-1}>0\quad(v\in R),\qquad
 \Lambda=\prod_{v\in R}\eta_v.
\]
Let \(\occ^+\) keep the left occupations and replace each right occupation
by \(b_v\). Thus \(1\le a_v\le b_v\le M-1\).

\begin{lemma}[Nonnegative saturation remainder]\label{lem:saturation}
There is an \(\varepsilon\ge0\) such that
\[
 D_{\widetilde G}(\occ)=\Lambda D_{\widetilde G}(\occ^+)+\varepsilon.
\]
\end{lemma}
\begin{proof}
At fixed selected left sets, the available set at right vertex \(v\)
is
\[
 A_v^{\mathrm{avail}}=
 \bigcap_{u\in N(v)}\alpha_{uv}([M]\setminus S_u),
 \qquad d_v=|A_v^{\mathrm{avail}}|\le b_v.
\]
On the entire domain \(0,\ldots,M\), define
\[
 f_v(d)=\binom{\min(d,b_v)}{a_v},\quad
 g_v(d)=\ind{d\ge b_v},\quad h_v=f_v-\eta_v g_v.
\]
All three functions are nonnegative and nondecreasing. At the only
possible problematic jump, Pascal's identity gives
\[
 h_v(b_v)=\binom{b_v-1}{a_v}=h_v(b_v-1);
\]
above it \(h_v\) is constant. For \(d\le b_v\), \(f_v\) is the original
cost and \(g_v(d)=\binom d{b_v}\) is the saturated cost.
Expanding \(\prod_v(h_v+\eta_vg_v)\) gives
\[
 D_{\widetilde G}(\occ)=
 \Lambda D_{\widetilde G}(\occ^+)+
 \sum_{\mathcal U\subsetneq R}\left(\prod_{v\in\mathcal U}\eta_v\right)\Delta_{\mathcal U}.
\]
To justify every remainder, use the \emph{transposed} bipartition:
rows are now the original right vertices and columns are the original
left vertices. Their fixed column sets are \([M]\setminus S_u\),
of ranks \(M-a_u\); their row intersections, with maps
\(\alpha_{uv}\), are exactly \(A_v^{\mathrm{avail}}\).
For each \(\mathcal U\subsetneq R\) the cost on these row cardinalities is
\[
 F_{\mathcal U}(\mathbf d)=\prod_{v\in\mathcal U}g_v(d_v)
                \prod_{v\in R\setminus\mathcal U}h_v(d_v)\in\CC_{|R|}.
\]
Thus \(\Delta_{\mathcal U}\), its ordinary-minus-gained sum, is nonnegative by
Proposition~\ref{prop:conecompare} and the transposed root-last order.
These costs are generally not binomial occupation costs; the whole-cone
statement of that proposition is essential here.
The nonnegative sum of these remainders is \(\varepsilon\).
The cap in \(f_v\) ensures cone membership on the full domain
\(0,\ldots,M\), even though actual available sizes satisfy \(d_v\le b_v\).
\end{proof}

\begin{lemma}[Tight-star contraction]\label{lem:star}
Choose a left vertex \(u\) of maximal occupation \(a_*\).
At \(\occ^+\), gauge its star edges to identity, delete \(u\), and identify
all vertices of \(N(u)\) to one right vertex \(z\) of occupation \(M-a_*\),
retaining all labelled edges and maps. This gives a multigraph \(H\)
with fewer vertices, whose underlying simple graph \(\overline H\) is chordal
bipartite. The assignment correspondence is a bijection in both models.
The full labelled system on \(H\) is flat if and only if that on \(G\)
is flat.
\end{lemma}
\begin{proof}
Every neighbor of \(u\) has saturated occupation \(M-a_*\), so its
selected set equals \([M]\setminus S_u\).
This common set becomes \(S_z\) and determines the removed \(S_u\).
All other constraints are retained, giving a bijection without a
binomial multiplier. The vertex number decreases by
\(|N(u)|\ge1\); the new profile is feasible because all other left
occupations are at most \(a_*\).

Suppose \(\overline H\) had an induced cycle of length at least six.
If it avoids \(z\), it persists in \(G\).
Otherwise let \(x,y\) be its left neighbors of \(z\).
Its remaining \(x\)-\(y\) path has right interior outside \(N(u)\),
and no interior left vertex is adjacent to any old member of \(N(u)\),
as that would be a chord to \(z\).
If \(x,y\) share a neighbor \(w\in N(u)\), replace \(z\) by \(w\)
to obtain an induced cycle in \(G\).
Otherwise choose respective neighbors \(w_x,w_y\in N(u)\) and replace
\(x-z-y\) by \(x-w_x-u-w_y-y\).
This longer cycle is induced: the old path has no chord, \(u\) has
no right neighbor on its interior, neither new right vertex has an
internal left neighbor, and cross-endpoint edges would give a common
neighbor. Both alternatives contradict chordal bipartiteness.

For a multigraph, flatness requires a gauge making \emph{every}
labelled edge identity, including all parallel edges.
In the identity-star gauge, any flattening gauge agrees at \(u\) and
all its neighbors and descends through their identification.
Conversely a flat gauge on \(H\) extends using that common gauge on
the whole old star. This proves equivalence.
\end{proof}

\begin{lemma}[Rejected parallel constraint]\label{lem:parallel}
Retain one edge from each parallel class of \(H\), giving a cover \(Y\)
of \(\overline H\). For the inherited profile,
\[
 P_G(\occ^+)=P_{\overline H},\quad Z_{\widetilde G}(\occ^+)=Z_H\le Z_Y,\qquad
 D_{\widetilde G}(\occ^+)=D_Y+(Z_Y-Z_H).
\]
If \(Y\) is flat and \(H\) is not, then \(Z_Y-Z_H\ge1\).
\end{lemma}
\begin{proof}
Ordinary parallel constraints are duplicates; gained ones may disagree.
This proves the identities and inequality.
Gauge \(Y\) to identity. Some omitted edge from left \(x\) to right
\(y\) has map \(\tau\ne1\). Put \(p=a_x\), \(q=M-a_y\), so
\(1\le p\le q\le M-1\).
Choose \(i\) moved to \(j\ne i\), a \(q\)-set \(Q\) containing \(i\)
but not \(j\), and a \(p\)-set \(A\subseteq Q\) containing \(i\).
Then \(A\subseteq Q\), but \(\tau A\not\subseteq Q\).
Order labels with \(A\) first, then \(Q\setminus A\), then the rest.
All left prefixes and right complementary prefixes in that order
give a globally feasible ordinary assignment, with selected sets
\(A,[M]\setminus Q\) at \(x,y\).
It violates the omitted constraint, proving strict rejection.
\end{proof}

\begin{proof}[Proof of rigidity in Theorem~\ref{thm:A}]
Comparison follows from Proposition~\ref{prop:conecompare}.
Flatness implies equality by Lemma~\ref{lem:flat}.
For the converse induct on the number of positive active vertices,
removing isolated vertices and their common positive factors first.
An empty remaining graph is flat; at \(M=1\) this is the only case.
In any other instance the ranks are interior and there is an edge.

If \(D_{\widetilde G}(\occ)=0\), Lemma~\ref{lem:saturation} gives
\(D_{\widetilde G}(\occ^+)=0\).
Saturation does not reduce the induction measure; contraction in
Lemma~\ref{lem:star} reduces it by \(|N(u)|\ge1\).
The representative \(\overline H\) remains chordal bipartite.
Any new isolated vertices are removed before induction, including
\(z\) if \(N(N(u))=\{u\}\); their factors are positive.
Induction gives \(D_Y\ge0\), with equality exactly for a flat \(Y\).
Lemma~\ref{lem:parallel} has two nonnegative terms whose sum is zero,
so \(Y\) is flat and \(Z_Y-Z_H=0\).
Its strict-rejection assertion forces \(H\) flat, and
Lemma~\ref{lem:star} extends flatness back to \(G\).
Combining the exact identities proves \eqref{eq:three}.
\end{proof}
This argument establishes equality rigidity. It does not supply
\(2\Ra_G(\occ)\) on the original chordal-bipartite graph: retaining only
one saturated branch may discard part of that reserve.

\section{Fibre comparison and composition}\label{sec:composition}

Gluing inequalities for DP counts have precedents. If \(H\) is formed
from \(H_1,\ldots,H_s\) by identifying one actual edge of each graph,
with otherwise disjoint interiors, Kaul, Maxfield, Mudrock, and Thomason
prove, for \(M\ge2\),
\[
 P_{DP}(H,M)\le
 \frac{\prod_{j=1}^s P_{DP}(H_j,M)}{\{M(M-1)\}^{s-1}}
\]
\cite[Theorem~4]{GluingDP}.
Here \(P_{DP}\) is the minimum over covers, and the selections are
singletons. Our comparison is uniform for each fixed cover and each
legal boundary state at arbitrary occupations. This fibre bound retains
the equality condition, prevents dilution of relative deficits, and
propagates a designated piece's bound to the original reduced reserve.

\begin{example}[A nonadjacent boundary]\label{ex:nonedge}
On the path \(u-v-w\), take \(M=2\) and occupation one everywhere.
Fix \(S_u=\{1\}\), \(S_w=\{2\}\). With identity transports the
middle vertex has no choice; if \(uv\) has identity transport and
\(wv\) has map \((12)\), then \(S_v=\{2\}\) is its unique choice.
The ordinary completion count is instead one when the two endpoint
pins coincide. Thus ordinary counts at a nonadjacent pair are not a
single boundary-independent constant, and the actual-edge fibre
comparison cannot be extended to arbitrary two-vertex boundaries.
\end{example}

\subsection{Amplification at an actual edge}
An edge is an \emph{actual edge} of the graph throughout this section,
not a pair of nonadjacent terminals. Normalize its transport to identity.
At fixed occupations, its legal raw boundary pairs are
\[
 \begin{gathered}
 \mathcal B=\{(A,B):|A|=a_u,\ |B|=a_v,\ A\cap B=\varnothing\},\\
 \beta=|\mathcal B|=\binom M{a_u}\binom{M-a_u}{a_v}.
 \end{gathered}
\]
For legal ranks \(\beta>0\). Simultaneous permutations are transitive
on \(\mathcal B\), so the ordinary completion count is one constant
\(c=\Pa_G/\beta\). Denote a gained completion count by \(K(s)\).

At a root vertex \(x\), a pin is a set \(S_x\) of size \(a_x\); its
ordinary fibre constant is \(\Pa_G/\binom M{a_x}\). A \emph{gained
fibre} counts completions of prescribed sets in the transported cover;
an \emph{ordinary fibre} counts them in the identity cover. Legal edge
pins before normalization satisfy \(\alpha_{uv}A\cap B=\varnothing\).
All ranks lie in \(0,\ldots,M\); for an edge the additional legality
condition is \(a_u+a_v\le M\). Fibre counts and constants are allowed
to be zero when interior occupations are infeasible.

\begin{mainthm}[Composition]\label{thm:C}
Let \(M\ge1\), and let \(H\) be a finite seam composition, as defined in
Section~\ref{sec:intro-main}, with an \(M\)-cover.
\begin{enumerate}
\renewcommand{\labelenumi}{(\roman{enumi})}
\item For ranks \(0\le a_v\le M\) and legal pins, every gained fibre at
a root or at any actual edge is bounded by
the corresponding ordinary fibre constant. Hence the coefficientwise
comparison holds on \(H\).
\item At every feasible profile, scalar equality holds exactly when the
active cover is flat. Moreover, writing
\(\delta_J=\Da_J/\Pa_J\) for the inherited cover and profile on a constructed
piece \(J\), one has \(\delta_H\ge\delta_J\) for every atom and intermediate
piece \(J\).
\item Suppose \(M\ge2\), the whole profile is positive and feasible, and a
designated atom occurrence \(W\) satisfies
\(\Da_W\ge\lambda\Ra_W\), where \(\lambda>0\).
Then \(\Da_H\ge\lambda\Ra_H\).
In particular a connected nonflat width-two witness supplies
\(\lambda=\kappa(\Gamma_W)\), even if the other atoms have arbitrary holonomy.
\end{enumerate}
In (iii), \(\Ra_H\) may be zero at an isolated vertex with occupation \(M\);
it is positive if every occupation is at most \(M-1\).
\end{mainthm}

\begin{lemma}[Copy amplification]\label{lem:amplify}
Suppose a graph class has scalar comparison for every cover and profile
and is closed under gluing arbitrarily many copies of any member along
one actual edge, with disjoint interiors. Then \(K(s)\le c\) at every
legal boundary pair. Both atom classes of Theorem~\ref{thm:C}
satisfy the closure hypothesis.
\end{lemma}
\begin{proof}
Glue \(k\) identical copies along the normalized edge. Its gained and
ordinary counts are \(\sum_sK(s)^k\) and \(\beta c^k\).
Comparison gives \(K(s)^k\le\beta c^k\).
If \(c=0\), use \(k=1\); otherwise take \(k\)-th roots and let
\(k\to\infty\), using \(\beta^{1/k}\to1\).

For width-two graphs join edge-covering bags of the individual
decompositions by a tree. Bags do not grow, and each common endpoint
has a connected occurrence set, proving closure for every \(k\).
For chordal-bipartite graphs, a cycle crossing copies must pass through
both shared endpoints. If it avoids the shared edge, that edge is a
chord. If it uses that edge, the other endpoint-to-endpoint path lies
in one copy. Hence every induced cycle of length at least six would
already occur within an atom, which is impossible.
Finally, the active graph of every profile on the repeated copies is
an induced subgraph of their union. Treewidth bounds and absence of long
induced cycles are both preserved by induced subgraphs.
Theorem~\ref{thm:A} therefore supplies scalar comparison for every such
profile, including profiles whose active graph omits a seam endpoint.
\end{proof}
Before normalization the legal pins satisfy
\(\alpha_{uv}A\cap B=\varnothing\).
The reference \(c\) is the coordinate-independent ordinary constant,
not an ordinary completion count at a possibly illegal unnormalized pair.

\subsection{Closure and equality}
Say a piece has the \emph{fibre property} if it has scalar comparison
and the ordinary upper bound at every actual edge and every root.

\begin{lemma}[Fibre closure]\label{lem:fibreclosure}
Both atom classes have the fibre property, and that property is preserved
by every seam in a finite seam composition.
\end{lemma}
\begin{proof}
Lemma~\ref{lem:amplify} gives each atom's edge fibres.
Fixing a selected root set \(S_x\) and summing the edge bound over its
\(\binom{M-a_x}{a_y}\) legal partner sets gives the ordinary root
constant \(\Pa/\binom M{a_x}\).
For an isolated root use scalar comparison on the remaining graph.

Let \(H_1,H_2\) share exactly an edge \(xy\), with consistent data,
and retain any desired edge \(e\), say in \(H_1\).
For each gained \(H_1\) assignment with the \(e\) pins fixed, the number
of \(H_2\) completions is at most its ordinary \(xy\)-fibre constant
\(c_2\). Thus
\[
 K_{H_1\cup H_2}(e\text{ pins})
 \le c_2K_{H_1}(e\text{ pins})\le c_2c_1(e).
\]
In the ordinary model each legal \(xy\) pair has exactly \(c_2\)
completions, so \(c_2c_1(e)\) is the union's ordinary fibre constant.
The shared edge indicator is \(0\) or \(1\); including its constraint
in each conditional count introduces no additional weight.
Summing fibres gives scalar comparison, and the root argument just
given supplies every root bound.
The same proof works at a one-vertex seam, using its root bound on the
second piece. Disjoint unions multiply, with the scalar bound on the
outside piece for any retained pins. All arguments include \(c_2=0\)
without division.
Finite induction on the construction proves closure.
\end{proof}

\begin{proposition}[Equality and seam identity]\label{prop:seamequality}
At a feasible actual-edge or one-vertex seam, let \(\beta\) be the
legal boundary cardinality, \(c_i>0\) the ordinary constants,
and \(K_i(s)\) the gained fibres. Then
\begin{equation}\label{eq:seam}
 D_{12}=\beta c_1c_2-\sum_sK_1(s)K_2(s)
       =c_2D_1+\sum_sK_1(s)(c_2-K_2(s)).
\end{equation}
Consequently \(D_{12}\ge c_2D_1\) and \(D_{12}\ge c_1D_2\).
Equality for the union holds if and only if both pieces have scalar
equality. The same equality equivalence holds for an empty seam.
On a finite seam composition it is equivalent to active flatness.
\end{proposition}
\begin{proof}
The first identity counts assignments conditional on their common pins.
The second follows by writing \(D_1=\beta c_1-\sum_sK_1(s)\).
Every remainder is nonnegative by Lemma~\ref{lem:fibreclosure};
the reversed identity gives the other lower bound.
Feasibility makes \(\Pa_i,c_i\) positive by
Lemma~\ref{lem:counting}. Thus union equality forces both \(D_i=0\).
Conversely \(D_i=0\), together with \(0\le K_i(s)\le c_i\), forces
\(K_i(s)=c_i\) at every pair, and gives union equality.
For an empty seam use the product of two positive ordinary counts.

Each atom has equality exactly for active flatness by Theorem~\ref{thm:A}.
Flat gauges on an active shared edge can be aligned at one endpoint
by right-composing one component's gauges with a common permutation.
The shared edge equation then aligns them at the other endpoint.
If zero occupations remove one endpoint, only one gauge must be aligned;
if both vanish, active pieces are disjoint.
A vertex seam requires the same single alignment.
Restriction proves the reverse flatness implication.
Induction on the construction establishes the assertion.
\end{proof}

\begin{corollary}[Relative deficit cannot decrease]\label{cor:relative}
For \(M\ge1\) and a feasible profile on a finite seam composition,
\(\delta_J=D_J/P_J\) is defined and lies in \([0,1]\).
At an edge or vertex seam,
\[
 \delta_{12}\ge\max\{\delta_1,\delta_2\}.
\]
At an empty seam,
\[
 \delta_{12}=\delta_1+\delta_2-\delta_1\delta_2
 \ge\max\{\delta_1,\delta_2\}.
\]
The full relative deficit is at least that of every atom and intermediate
constructed piece. Zero occupations and isolated full ranks are allowed.
\end{corollary}
\begin{proof}
For a nonempty seam, ordinary factorization gives
\(P_{12}=c_2P_1=c_1P_2>0\).
Divide the two bounds following \eqref{eq:seam} by \(P_{12}\).
For an empty seam \(Z_{12}/P_{12}=(Z_1/P_1)(Z_2/P_2)\),
which gives the formula and inequality since \(0\le\delta_i\le1\).
Apply each one-step conclusion up the finite construction tree.
\end{proof}
For an infeasible profile \(P=Z=0\), this relative quantity is undefined.

\subsection{Propagation to the original reduced coefficient}

\begin{proposition}[Witness propagation]\label{prop:witness}
Under the hypotheses of Theorem~\ref{thm:C}(iii),
\(D_H\ge\lambda R_H\). Disconnected partners and isolated vertices
are permitted.
\end{proposition}
\begin{proof}
We first prove one step, with \(H_1\) containing the designated witness
and already satisfying \(D_1\ge\lambda R_1\).
The other child \(H_2\) has the fibre property.
If \(R_{12}=0\), scalar comparison finishes the step.
Otherwise reduced feasibility gives
\(1\le a_v\le M-1\) at \emph{every} vertex, including isolated ones.
All original inclusion ranks are therefore interior.

At an edge seam let \(\beta'>0\) be the reduced boundary cardinality
and \(c_i'\) the reduced ordinary fibre constants. Then
\[
 R_1=\beta'c_1',\qquad R_{12}=\beta'c_1'c_2'.
\]
We claim \(c_2\ge c_2'\).
Prescribe original inclusion pins \(A\subseteq B\) with
\(1\le|A|\le|B|\le M-1\).
Choose \(i\in A\), \(j\notin B\), and prescribe
\(A\setminus\{i\}\subseteq B\setminus\{i\}\) on the remaining labels.
Adding \(i\) to every reduced inclusion set, and omitting \(j\), injects
reduced ordinary completions into original ordinary completions with
the desired pins. It works as well in disconnected components and
at interior isolated slices. Slice transitivity identifies the counts
with \(c_2'\) and \(c_2\).
Equation~\eqref{eq:seam} now gives
\[
 D_{12}\ge c_2D_1\ge\lambda c_2R_1
              \ge\lambda R_{12}.
\]

At a vertex seam choose \(i\) inside its prescribed inclusion set and
\(j\) outside, and use the same injection.
Here \(\beta'=\binom{M-2}{a_{\rm root}-1}>0\); the same factorization
and seam identity apply.
An isolated seam vertex of occupation \(M\) has zero reserve and was
handled before any division by a reduced boundary cardinality.
At an empty seam unpinned ordinary insertion gives \(P_2\ge R_2\);
hence
\[
 D_{12}=P_2D_1+Z_1D_2\ge P_2D_1
       \ge\lambda R_1R_2=\lambda R_{12}.
\]
An isolated full occupation anywhere makes its reduced factor zero,
so also belongs to the first case.

Finally use the assumed bound at the designated witness leaf.
At each ancestor in the construction tree exactly one child contains
that leaf; apply its inductive bound and the other child's fibre
property in the one-step argument.
The finite path to the root proves the result.
Reduced feasibility also gives \(R_H>0\) when all original occupations
are at most \(M-1\).
\end{proof}

\begin{proof}[Proof of Theorem~\ref{thm:C}]
Part (i) is Lemma~\ref{lem:fibreclosure}.
Part (ii) follows from Proposition~\ref{prop:seamequality} and
Corollary~\ref{cor:relative}. Part (iii) is
Proposition~\ref{prop:witness}, with Theorem~\ref{thm:B} supplying
the indicated width-two witness constant.
\end{proof}

\begin{example}[A mixed family beyond the atom classes]\label{ex:mixed}
Glue \(K_{n,n}\), \(n\ge3\), and \(C_6\) along one actual edge.
The complete bipartite atom is chordal bipartite: every cycle of length
at least six has at least three vertices on each side and a chord.
For \(C_6=(v_0,\ldots,v_5)\), the bags
\(\{v_0,v_i,v_{i+1}\}\), \(i=1,\ldots,4\), in a path give width two.
The union retains an induced \(C_6\) and hence is not chordal bipartite.
Its treewidth is \(n\). For the lower bound its \(K_{n,n}\) subgraph
has minimum degree \(n\), so Lemma~\ref{lem:lowdegree} rules out width
below \(n\). For the upper bound use bags consisting of all \(n\)
left vertices and one right vertex for \(K_{n,n}\), in a path,
and attach the cycle decomposition at edge-covering bags.
No bag grows. Thus the union is outside both atom classes but satisfies
Theorem~\ref{thm:C}.
\end{example}

\section{Questions left by the quantitative comparison}\label{sec:questions}

The general coefficientwise conjecture asks for comparison beyond the
composition class above. Equality of the \emph{whole polynomial} is already
settled by Proposition~\ref{prop:first}, but it does not classify equality
at an arbitrary fixed profile.
A further question is whether the original-reserve bound
\(D_{\widetilde G}(\occ)\ge2R_G(\occ)\) holds on every connected
chordal-bipartite base with a positive feasible profile.
The nonnegative decomposition \eqref{eq:three} proves strictness there;
paying the full original reserve would require quantitative control of
all its terms, not just one saturated branch.
The witness theorem avoids that unresolved estimate by retaining an
independently quantified width-two piece.

Can one elimination principle encompass both proofs? The serial
method keeps nondecreasing overlap kernels between two surviving slices;
nest-column compression keeps nondecreasing costs of several row
intersections. A common invariant would have to preserve this broader
cost information as well as the obstruction to flatness.

It is also natural to ask about unions along two nonadjacent vertices
or a larger boundary. Example~\ref{ex:nonedge} rules out the
boundary-independent ordinary fibre constant used here, but does not
rule out scalar coefficient comparison for such unions.
An extension would require a comparison that retains the dependence
on the boundary state before the final sum.

\appendix
\section{Counting-normalized slice calculus}\label{app:spectral}

This appendix proves Proposition~\ref{prop:scalars} and the decomposition
used in \eqref{eq:localweights}. Harmonic analysis on Boolean slices is
classical; see Filmus \cite[Theorem~3.1]{Filmus} for an
orthogonal basis and Ghareghani, Ghorbani, and Mohammad-Noori
\cite[Lemma~8]{Ghareghani} for the inclusion-matrix spectrum.
We rederive the needed formulas to fix the normalization.
Its purpose is to obtain common \emph{nonincreasing} nonnegative scalars;
that monotonicity makes the summation-by-parts weights nonnegative.
All inner products below are counting sums, not probability averages.

\subsection{Up and down operators}
The space \(V_k=\mathbb R^{\Om_k(M)}\) uses the inner product
\(\langle f,g\rangle=\sum_A f(A)g(A)\). Define
\[
 (U_kf)(B)=\sum_{\substack{A\subset B\\|A|=k}}f(A),\quad
 \partial_{k+1}=U_k^*,\quad \mathcal H_i=\ker \partial_i,\quad \partial_0=0.
\]
Interpret the missing \(U_{-1}\partial_0\) as zero.
Off-diagonal one-element exchanges appear once in both products;
the diagonal appears \(M-k\) and \(k\) times. Hence
\begin{equation}\label{eq:commutator}
 \partial_{k+1}U_k-U_{k-1}\partial_k=(M-2k)I.
\end{equation}
For \(h\in \mathcal H_i\), induction using this identity gives
\[
 \partial_{i+j}U_{i+j-1}\cdots U_i h
 =j(M-2i-j+1)U_{i+j-2}\cdots U_i h,
\]
where the chain on the right is the identity when \(j=1\).
Taking inner products recursively yields
\[
 \|U_{i+j-1}\cdots U_i h\|^2
 =(j!)^2\binom{M-2i}{j}\|h\|^2.
\]
Write
\[
 E_{k,i}=\frac1{(k-i)!}U_{k-1}\cdots U_i,\qquad
 E_{k,i}h(A)=\sum_{\substack{S\subseteq A\\|S|=i}}h(S).
\]
For \(i\le k\le M-i\), this has squared norm multiplier
\(\binom{M-2i}{k-i}\). Thus
\(J_{k,i}=\binom{M-2i}{k-i}^{-1/2}E_{k,i}\)
is an isometry.

For \(k\le M/2\), induction shows \(U_{k-1}\) is injective on
every preceding harmonic block, because its norm multiplier is positive.
Hence \(\partial_k\) is onto and
\(V_k=\operatorname{im}U_{k-1}\oplus\ker \partial_k\).
Different lifted harmonic blocks are orthogonal: move up operators to
the other side as down operators until a down operator kills the
higher-degree harmonic vector.
For \(k>M/2\), the surviving blocks \(i\le M-k\) remain orthogonal
and their dimensions telescope to \(\binom M{M-k}=\dim V_k\).
Consequently
\begin{equation}\label{eq:harmonic}
 \begin{gathered}
 V_k=\bigoplus_{i=0}^{\min(k,M-k)}J_{k,i}\mathcal H_i,\\
 d_i:=\dim \mathcal H_i=\binom Mi-\binom M{i-1}
 \quad(0\le i\le\lfloor M/2\rfloor).
 \end{gathered}
\end{equation}
The formula uses \(\binom M{-1}=0\).
Every operator commutes with label permutations. Therefore a lifted
block carries the same representation as \(\mathcal H_i\); no multiplicity-one
representation theorem is being imported.

\subsection{Common decreasing scalars}
Let \(W_{k,d}(A,S)=\ind{S\subseteq A}\), as a map \(V_d\to V_k\).
Counting intermediate \(d\)-sets gives
\[
 W_{k,d}E_{d,i}=\binom{k-i}{d-i}E_{k,i}.
\]
For \(0\le d\le k\le M\) and \(0\le i\le\min(d,M-k)\),
normalizing this identity explicitly yields
\[
 W_{k,d}J_{d,i}
 =\binom{k-i}{d-i}
  \sqrt{\frac{\binom{M-2i}{k-i}}{\binom{M-2i}{d-i}}}\,J_{k,i}
 =\sqrt{A_i(M;k,d)}\,J_{k,i}.
\]
Thus the block action follows by direct counting and normalization,
without a multiplicity-one representation premise.
On normalized blocks its squared scalar is
\begin{equation}\label{eq:Ai}
 A_i(M;k,d)=
 \binom{k-i}{d-i}\binom{M-d-i}{k-d}.
\end{equation}
Indeed the unsimplified value is
\(\binom{k-i}{d-i}^2
 \binom{M-2i}{k-i}/\binom{M-2i}{d-i}\);
factorial cancellation gives \eqref{eq:Ai}.
Outside the surviving block range the scalar is zero.
In that range, \eqref{eq:Ai} is the inclusion-spectrum formula attributed
to Wilson in \cite[Lemma~8]{Ghareghani}: its parameters
\((v,i,j)\) correspond to \((M,d,i)\) here.

First assume \(p+q\le M\).
For the binomial overlap matrix
\(B_d(A,B)=\binom{|A\cap B|}{d}\),
\[
 \begin{aligned}
 B_d&=W_{p,d}W_{q,d}^*
 =\sum_i\beta_i(d;p,q)J_{p,i}J_{q,i}^*,\\
 \beta_i(d;p,q)&=\sqrt{A_i(M;p,d)A_i(M;q,d)}\ge0.
 \end{aligned}
\]
Here \(d\le\min(p,q)\le M/2\), so all blocks of \(V_d\) are covered
by \eqref{eq:harmonic}.
For consecutive nonzero terms,
\[
 \frac{A_{i+1}(M;k,d)}{A_i(M;k,d)}
 =\frac{(d-i)(M-k-i)}{(k-i)(M-d-i)}\le1;
\]
the denominator minus the numerator is \((k-d)(M-2i)\ge0\).
Vanishing blocks are final zeros.
Every \(f=\sum_dc_d\binom zd\) consequently has the same intertwiners
and nonnegative nonincreasing scalars
\[
 K_f=\sum_i k_i(f)J_{p,i}J_{q,i}^*,\qquad
 k_i(f)=\sum_dc_d\beta_i(d;p,q).
\]
For \(p+q>M\), complement both endpoint sets.
The shifted overlap becomes the unshifted overlap on complementary
ranks whose sum is at most \(M\).
Complementation is an orthogonal, permutation-equivariant bijection,
so the common-intertwiner and decreasing-scalar conclusions still hold.

\subsection{Characters and cumulative moved-subset counts}
Let \(\chi_i(\gamma)\) be the trace of \(\gamma\) on \(\mathcal H_i\).
The two Hilbert--Schmidt products in the definition of \(S_\gamma\)
are the traces of \(K_fK_g^*\) and of this operator followed by the
endpoint permutation. Equivariance and \eqref{eq:harmonic} therefore give
\[
 S_\gamma=\sum_{i=1}^r\mu_i(d_i-\chi_i(\gamma)),\qquad
 \mu_i=k_i(f)k_i(g),\quad r=r_{pq}.
\]
The degree-zero term vanishes.
The inverse action convention has the same character, since a real
orthogonal matrix and its inverse have the same trace.

A permutation matrix on \(V_k\) has trace equal to the number of fixed
\(k\)-subsets. For \(k\le M/2\), \eqref{eq:harmonic} hence implies
\[
 m_k(\gamma)=\sum_{i=1}^k(d_i-\chi_i(\gamma)).
\]
With \(\mu_{r+1}=0\), summation by parts proves
\[
 S_\gamma=\sum_{k=1}^r(\mu_k-\mu_{k+1})m_k(\gamma),
\]
which is \eqref{eq:localweights}. Both scalar sequences are nonnegative
and nonincreasing, so every weight is nonnegative.
Only these cumulative moved-subset counts are compared.
No termwise lower bound on individual harmonic character deficits
is required.

\section{A self-contained root-last structural order}\label{app:structure}

The standard structural assertions in this appendix are included to
make Proposition~\ref{prop:conecompare} self-contained.
The two-nest-point argument follows Brault-Baron's two-beta-leaf proof
\cite[Theorem~7 and its proof]{BraultBaron}, with the explicit
vertex-set convention below. Choosing a nonroot nest point at each step
then gives the root-last refinement.
The underlying nest-point theorem goes back to Brouwer and Kolen
\cite{BrouwerKolen}; the hypergraph and incidence-graph characterizations
are classical \cite{Fagin,Ausiello}. We include this consequence for the
rooted induction, without asserting a new structural result.
A hypergraph has an explicit finite vertex set; its nonempty edges
may form an empty family. Its incidence graph joins a column vertex
to an edge row exactly when the vertex belongs to that edge.
Duplicate rows can be discarded. Adding singleton edges for every
vertex neither creates an induced incidence cycle nor changes the
nesting condition at a vertex.

\subsection{Simplicial vertices of chordal graphs}
A graph is chordal if it has no induced cycle of length at least four.
A vertex is simplicial if its neighbors form a clique.
We prove by vertex induction that a noncomplete chordal graph with at
least two vertices has two nonadjacent simplicial vertices; every vertex
of a complete graph is simplicial.
For a disconnected graph choose a simplicial vertex in each of two
components by induction.
For a connected noncomplete graph choose nonadjacent \(x,y\), an
inclusion-minimal \(x\)-\(y\) separator \(S\), and their components
\(C_x,C_y\) after deletion of \(S\).
Every \(s\in S\) has a neighbor in both components: a path after deleting
only \(S\setminus\{s\}\) must run through \(s\) between them.
If \(s,t\in S\) were nonadjacent, take shortest \(s\)-\(t\) paths with
interiors in \(C_x\) and \(C_y\).
They are induced, have length at least two, and their interiors have
no cross-edges. Their union is an induced cycle of length at least four.
Thus \(S\) is a clique.

The induced graphs on \(C_x\cup S\), \(C_y\cup S\) are smaller.
If one is complete choose a vertex of its component.
If it is not, induction gives two nonadjacent simplicial vertices,
at most one in the clique \(S\); choose the other.
It has no neighbor outside its component and \(S\), so is simplicial
in the whole graph. The two chosen vertices are nonadjacent.
This closes the induction.

\subsection{An alpha leaf from incidence cycle freedom}
Suppose a hypergraph has a nonempty vertex set and its incidence graph
has no induced cycle of length at least six.
Retain the singleton edges just described.
Its primal graph joins vertices lying in a common hyperedge.
This primal graph is chordal: an induced primal cycle of length
\(k\ge4\) has a hyperedge witnessing each consecutive pair.
A witnessing edge contains no other cycle vertex, which would create
a primal chord. The witnessing edges are distinct, and with the \(k\)
columns they form an induced incidence cycle of length \(2k\),
a contradiction.

Every primal clique is contained in a hyperedge.
Otherwise take a smallest uncovered clique \(K\).
Singletons and adjacent pairs are covered, so \(|K|\ge3\).
For each \(z\in K\) some edge \(e_z\) contains \(K\setminus\{z\}\)
and avoids \(z\). On any three vertices \(x,y,z\), the three rows
\(e_x,e_y,e_z\) form an induced incidence \(C_6\), a contradiction.

A simplicial primal vertex \(v\) now has its closed neighborhood
contained in a hyperedge \(e_v\). Since that edge contains \(v\),
all its vertices lie in the closed neighborhood, giving equality.
Every edge incident to \(v\) lies inside \(e_v\).
Thus \(v\) has a largest incident edge, an \emph{alpha leaf}.

\subsection{Two nest points}
Following the full-edge removal and alpha-leaf induction in
\cite[proof of Theorem~7]{BraultBaron}, we prove the stronger
assertion: a hypergraph with at least two vertices
and no induced long incidence cycle has two nest points nonadjacent
after its possible full edge is removed.
Keep singletons, collapse duplicates, and induct on
\(|V|+\sum_e|e|\).
If \(V\) itself is an edge, remove it and apply induction on the same
vertex set. Reinserting the full edge preserves each incident chain,
and the two points are nonadjacent after that full edge is removed.

Assume no full edge. With two vertices only singleton edges remain.
For \(|V|\ge3\) choose an alpha leaf \(x\) with largest edge \(e_x\subsetneq V\).
Delete \(x\), take traces, and discard empty and duplicate edges.
Incidence cycle freedom persists, since this is column deletion followed
by harmless row removals. By induction the smaller hypergraph has
two nest points \(y,z\), nonadjacent after its full edge is removed.
The edge \(e_x\setminus\{x\}\) is proper in \(V\setminus\{x\}\),
so cannot contain both; if it is empty, both are outside.
Choose \(y\notin e_x\setminus\{x\}\).
No old edge contains \(x,y\); every edge incident to \(y\) is unchanged
by deleting \(x\). Hence \(y\) is a nest point of the original hypergraph.

This argument gives, from any alpha leaf, a nest point outside its
largest edge. A nest point is itself an alpha leaf.
Apply the same deletion argument to \(y\), using induction on the
smaller hypergraph, to obtain a second nest point nonadjacent to \(y\).
This proves the stronger assertion and the induction.

\subsection{Root-last orders}
For a chordal-bipartite \(G=(L,R,E)\), its row neighborhoods form the
hypergraph on explicit vertex set \(R\), with incidence graph \(G\)
after harmless empty-row handling.
If at least two columns remain, the two-nest-point result provides
a nest point different from a prescribed root \(\rho\).
Delete it and repeat; isolated columns are vacuous nest points.
This produces a beta order ending at \(\rho\), or an unrooted order
when no root is specified. Empty vertex sets need no leaf.

Conversely, an induced incidence cycle of length \(2k\ge6\) has
\(k\) columns and \(k\) distinct rows, each row meeting those columns
in its consecutive pair. The first such column in a beta order has
two incident rows with incomparable suffix neighborhoods: their
other cycle columns have not yet been removed.
This contradicts nesting. Thus a beta order excludes all such cycles.
Transposing the bipartition proves the other-side assertion, and
vertex or row deletion preserves the cycle exclusion.
This completes the proof of Lemma~\ref{lem:rootlast}.

\end{document}